\documentclass[a4paper,11pt,leqno]{amsart}

\usepackage[dvips]{graphics}
\usepackage[matrix,arrow,tips,curve]{xy}
\usepackage[english]{babel}
\usepackage{amsmath}
\usepackage{amssymb,amsthm}

\usepackage{mathrsfs}

\usepackage{enumerate}
\usepackage{psfrag}
 \psfrag{y}{\fontsize{20}{20}$y$}
 \psfrag{C0}{\fontsize{20}{20}$C_0=G_Y\cap \mathbb{F}$}
 \psfrag{p1}{\fontsize{20}{20}$p_1$}
 \psfrag{p2}{\fontsize{20}{20}$p_2$}
 \psfrag{q1}{\fontsize{20}{20}$q_1$}
 \psfrag{q2}{\fontsize{20}{20}$q_2$}
 \psfrag{C}{\fontsize{20}{20}$C_1$}
 \psfrag{G1}{\fontsize{20}{20}$\wi{G}_Y\cap \mathbb{F}$}
\psfrag{Y}{\fontsize{20}{20}$Y$}
\psfrag{X}{\fontsize{20}{20}$X$}
\psfrag{n}{\fontsize{20}{20}$\sigma$}

\psfrag{f}{\fontsize{20}{20}$\longrightarrow$}
\psfrag{G}{\fontsize{20}{20}$G$}
\psfrag{Gh}{\fontsize{20}{20}$\wi{G}$}
\psfrag{Gy}{\fontsize{20}{20}$G_Y$}
\psfrag{Ghy}{\fontsize{20}{20}$\wi{G}_Y$}
\psfrag{E}{\fontsize{20}{20}$E$}
\psfrag{Eh}{\fontsize{20}{20}$\wi{E}$}
\psfrag{s}{\fontsize{20}{20}$\sigma(\wi{E})=\pi^{-1}(A)$}
\psfrag{A}{\fontsize{20}{20}$A_Y$}

\newtheorem{thm}{Theorem}[section]
\newtheorem*{thm*}{Theorem}
\newtheorem{lemma}[thm]{Lemma}
\newtheorem{proposition}[thm]{Proposition}
\newtheorem{corollary}[thm]{Corollary}

\newtheorem{construction}[thm]{Construction}

\theoremstyle{definition}

\newtheorem{remark}[thm]{Remark}

\newtheorem{example}[thm]{Example}
\newtheorem{parg}[thm]{}

\newtheorem{pargtwo}{}[thm]

\renewcommand{\theequation}{\thethm}

\newcommand{\ph}{\varphi}
\newcommand{\w}{\widetilde}
\newcommand{\wi}{\widehat}
\newcommand{\pr}{\mathbb{P}}
\newcommand{\Q}{\mathbb{Q}}
\newcommand{\R}{\mathbb{R}}
\newcommand{\N}{\mathcal{N}_1}
\newcommand{\ol}{\mathscr{O}}
\newcommand{\Nu}{\mathcal{N}^1}

\newcommand{\Obs}{\operatorname{Obs}}
\newcommand{\Spec}{\operatorname{Spec}}
\newcommand{\Sing}{\operatorname{Sing}}
\newcommand{\NE}{\operatorname{NE}}
\newcommand{\Exc}{\operatorname{Exc}}
\newcommand{\Lo}{\operatorname{Locus}}

\newcommand{\Hilb}{\operatorname{Hilb}}
\newcommand{\Chow}{\operatorname{Chow}}
\newcommand{\Pic}{\operatorname{Pic}}

\newcommand{\Sec}{\operatorname{Sec}}
\newcommand{\Proj}{\operatorname{Proj}}
\newcommand{\Sym}{\operatorname{Sym}}
\newcommand{\Rat}{\operatorname{RatCurves}^n}

\newcommand{\Hom}{\textup{Hom}}
\newcommand{\Ext}{\textup{Ext}}

\newlength{\Mheight}
\newlength{\cwidth}
\title[]{Locally unsplit families of rational curves of large anticanonical degree on Fano manifolds}
\author[]{Cinzia Casagrande}
\address{Cinzia Casagrande:
 Universit\`a di Torino,
 Dipartimento di Matematica,
via Carlo Alberto 10,
 10123 Torino - Italy}
\email{cinzia.casagrande@unito.it}

\author[]{St{\'e}phane Druel}
\address{St{\'e}phane Druel:
Institut Fourier,
UMR 5582 du CNRS,
Universit\'e Grenoble 1, BP 74,
38402 Saint Martin d'H\`eres - France}
\email{druel@ujf-grenoble.fr}
\date{January 3, 2015}
\subjclass[2010]{14J45, 14C05, 14E30}

\begin{document}
\maketitle

\begin{abstract}
In this paper we address Fano manifolds of dimension $n\ge 3$
with a locally unsplit dominating family  of rational curves of anticanonical degree $n$. We first observe that their 
Picard number is at most $3$, and then we provide a classification of all cases with maximal Picard number.
We also give examples of locally unsplit dominating families of rational curves whose varieties of minimal tangents at a general point are singular.
\end{abstract}

{\small\tableofcontents}

\section{Introduction}
\noindent 
Let $X$ be a Fano manifold, and let $V$ be a dominating family of rational curves  on $X$. By this we mean that $V$ is an irreducible component of 
 $\Rat(X)$, the scheme parametrizing 
integral rational curves on $X$, and that the union of the curves parametrized by $V$ is dense in $X$.

We say that $V$ is \emph{locally unsplit} if for general $x\in X$, the subfamily $V_x$ of curves containing $x$ is proper.  This is true, for instance, if $V$ is a dominating family with minimal degree with respect to some ample line bundle on $X$.

When $V$ is locally unsplit, the anticanonical degree of the curves of the family can vary between $2$ and $n+1$, where $n$ is the dimension of $X$.

Following Miyaoka \cite{miyaokaQ},
we define $l_X$
to be the minimal anticanonical degree of a locally unsplit dominating family of rational curves in $X$, so that $l_X\in\{2,\dotsc,n+1\}$. Equivalently, $l_X$ is the minimal anticanonical degree of a free rational curve in $X$ (see Remark \ref{deflX}). 

In this paper we study 
 Fano manifolds
 with a locally unsplit dominating family  of rational curves of anticanonical degree $n$, including in particular Fano manifolds $X$  with $l_X=n$.

Let us recall the following results, due respectively to Cho, Miyaoka, and Shepherd-Barron
(see also \cite{kebekus}), and to Miyaoka.
\begin{thm}[\cite{cho}]\label{CMSBK}
Let $X$ be a Fano manifold of dimension $n$. The following properties are equivalent:
\begin{enumerate}[$(i)$]
\item
 $X$ has a locally unsplit dominating family of rational curves of maximal anticanonical degree $n+1$;
\item
 $X\cong\pr^n$.
\end{enumerate}
 In particular, $l_X=n+1$ if and only if $X\cong\pr^n$.
\end{thm}
\begin{thm}[\cite{miyaokaQ}]\label{M}
Let $X$ be a Fano manifold of dimension $n\geq 3$, and with Picard number $\rho_X=1$. Then $l_X=n$ if and only if
$X$ is isomorphic to a quadric.
\end{thm}
On the other hand, there are also cases where
  $l_X=n$ and $\rho_X>1$.
\begin{example}[\cite{miyaokaQ}, Remark~4.2]\label{exampleM}
Let $A\subset \pr^n$ be a smooth subvariety, of dimension $n-2$ and degree $d\in\{1,\dotsc,n\}$, contained in a hyperplane. Let $X$ be the blow-up of $\pr^n$ along $A$. Then $X$ is Fano with $\rho_X=2$ and $l_{X}=n$. The locally unsplit dominating family of rational curves of anticanonical degree $n$ is given by the strict transforms of lines intersecting  $A$ in one point.
\end{example}
First of all, we show that in fact these are the only examples.
\begin{thm}\label{result}
Let $X$ be a Fano manifold of dimension $n\geq 3$, with $\rho_X>1$ and 
 $l_X=n$. Then there exists a smooth subvariety  $A\subset \pr^n$ of dimension $n-2$ and degree $d\in\{1,\dotsc,n\}$, contained in a hyperplane, such that
 $X$ is isomorphic to the blow-up of $\pr^n$ along $A$.
\end{thm}
Together with Miyaoka's result (Theorem \ref{M}), this 
gives a complete classification of Fano manifolds with $l_X=n$.

\medskip

Then we turn to the case where $X$ is a Fano manifold  having a locally unsplit dominating family  of rational curves of anticanonical degree $n$, where $n\geq 3$ is the dimension of $X$. Notice that this assumption is easier to check than the condition $l_X=n$, as it involves only one family of rational curves.

In the toric case, these varieties have been classified by Chen, Fu, and Hwang.
\begin{proposition}[\cite{fuhwang}, Proposition 3.8]
Let $X$ be a toric Fano manifold, of dimension $n\geq 3$, having a locally unsplit dominating family  of rational curves of anticanonical degree $n$. Then $X$ is one of the following:
\begin{enumerate}[1)]
\item  the blow-up of $\pr^n$ at a linear $\pr^{n-2}$ (here $\rho=2$ and $l_X=n$);
\item
 $\pr^1\times\pr^{n-1}$ (here $\rho=2$ and $l_X=2$); 
\item the blow-up of $\pr^n$  at $A\cup\{p\}$, where $A$ is a linear 
$\pr^{n-2}$, and $p$ a point not in $A$ (here $\rho=3$ and $l_X=2$).
\end{enumerate}
\end{proposition}
We  show that if $X$ has a locally unsplit dominating family $V$ of rational curves of anticanonical degree $n$, then the Picard number of $X$ is at most $3$ (see Proposition \ref{prop_main}). Moreover we classify all cases with $\rho_X=3$, giving a complete description of $X$ and $V$. 
Let us describe our results.

We first construct and study a family of examples.
\begin{example}\label{ex}
Fix integers $n,a$ and $d$ such that  
$n\geq 3$, $d\geq 1$ 
 and
$0\leq a\leq d$. Let moreover $A\subset \pr^{n-1}$ be a smooth hypersurface of degree $d$.

Set $Y:=\pr_{\pr^{n-1}}(\ol_{\pr^{n-1}}\oplus\ol_{\pr^{n-1}}(a))$, and
let $\wi{G}_Y\cong\pr^{n-1}\subset Y$ 
be a section of the $\pr^1$-bundle $Y\to\pr^{n-1}$ with normal bundle 
 $\mathcal{N}_{\wi{G}_Y/Y}\cong\ol_{\pr^{n-1}}(a)$.

Finally set $A_Y:=\wi{G}_Y\cap\pi^{-1}(A)$ (so that $A_Y\cong A$), and let $\sigma\colon X\to Y$ be the blow-up of $A_Y$.

Then $X$ is smooth of dimension $n$ and Picard number $3$, and it is Fano if and only if $a\leq n-1$ and $d-a\leq n-1$. In the Fano case, these varieties appear in \cite{toru}, where the author classifies Fano manifolds containing a divisor 
$G\cong\pr^{n-1}$ and with negative normal bundle.
\end{example}
\begin{proposition}\label{examples}
Let $X$ be as in Example \ref{ex}. Then
 $X$ has a locally unsplit dominating family $V$ of rational curves of anticanonical degree $n$.
\end{proposition}
Then we show that these are all the examples with $\rho\geq 3$.
\begin{thm}\label{main}
Let $X$ be a Fano manifold of dimension $n\geq 3$, and suppose that $X$ has a locally unsplit dominating family $V$ of rational curves of anticanonical degree $n$. Then $\rho_X\leq 3$. 

If moreover $\rho_X=3$, then $X$ is isomorphic to one of the varieties described in Example \ref{ex}, and the family $V$ is unique.
\end{thm}
We study in more detail the family of curves given by Proposition \ref{examples}. For $x\in X$,  we denote by $V_x$ the normalization of the closed subset of $V$ parametrizing curves containing $x$.
\begin{thm}\label{V_x}
Let $X$ and $V$ be as in Proposition \ref{examples}, and $x\in X$ a general point. Let $z\in\pr^{n-1}$ be the image of $x$ under the morphism $X\to\pr^{n-1}$, and  let 
$p_z\colon A \to \mathbb{P}^{n-2}$ be the degree $d$ morphism  induced by the linear projection $\mathbb{P}^{n-1}\dashrightarrow \mathbb{P}^{n-2}$ from $z$.

Then 
$V_x$ is smooth and connected. If $a=0$, then $V_x\cong\pr^{n-2}$. If $a>0$, then $V_x$ is isomorphic to the relative Hilbert scheme $\textup{Hilb}^{[a]}(A/\mathbb{P}^{n-2})$ of zero-dimensional subschemes, of length $a$, of fibers of $p_z$.
\end{thm}
Finally, we consider the  variety of minimal rational tangents (VMRT) associated to the locally unsplit family $V$ at a general point $x$, defined as follows. Let
 $$\tau_x\colon V_x \dasharrow \mathbb{P}(T_{X,x}^*)$$ be the map defined by sending a general curve from $V_x$ to its tangent direction at $x$, and
define the VMRT $\mathcal{C}_x$ to be the closure of image of $\tau_x$ in $\mathbb{P}(T_{X,x}^*)$.
We still denote by $\tau_x$ the induced map $V_x \to \mathcal{C}_x$; this is in fact the normalization morphism by \cite{kebekusJAG} and \cite{hwang_mok_birationality}.
\begin{thm}\label{tau_x}
Let $X$ and $V$ be as in Proposition \ref{examples}, and $x\in X$ a general point.
\begin{enumerate}[1)]
\item
The VRMT $\mathcal{C}_x\subset\mathbb{P}(T_{X,x}^*)$ is an irreducible hypersurface of degree $\binom{d}{a}$.
\item
If $a\in\{0,1,d-1,d\}$, then $\tau_x\colon V_x\to \mathcal{C}_x$ is an isomorphism.
\item
If $2\leq a\leq d-2$, then $\tau_x\colon V_x\to \mathcal{C}_x$ is not an isomorphism. More precisely: the closed subset where $\tau_x$ is not an isomorphism has codimension $1$, and the closed subset where $\tau_x$ is not an immersion has codimension $2$.
\end{enumerate}
\end{thm}
This provides the first examples of locally unsplit dominating families of rational curves whose VMRT at a general point $x$ is singular, equivalently such that $\tau_x\colon V_x\to\mathcal{C}_x$ is not an isomorphism (see \cite[Question 1]{hwangICTP}, \cite[Problem 2.20]{kebekussola}, \cite{hwang_equivalence}, and \cite{hwangcodone}). 
Notice that if $2\le a\le d-2$, then $V$ is not a dominating family of rational curves of minimal degree (see Remark \ref{minimaldegree}).

In \cite{hwangcodone}, Hwang studies projective manifolds $X$ having a locally unsplit
dominating family $V$ of rational curves of anticanonical degree equal to
the dimension of $X$. Under the assumption that the VMRT of $V$ at a
general point is smooth, he gives a birational description of $X$, see
\cite[Theorem 1.4]{hwangcodone}.

\medskip

In order to prove Theorems \ref{V_x} and \ref{tau_x}, we are led to study relative Hilbert schemes of zero-dimensional subschemes of the projection of a smooth hypersurface from a general point. We obtain the 
following result, of independent interest.
\begin{thm}\label{app}
Fix integers $m$, $a$, and $d$, such that $m\geq 1$ and 
$1\leq a\leq d$. Let $A\subset \pr^{m+1}$
be a smooth hypersurface of degree $d$.
Let $z\in \pr^{m+1}$ be a general point, and let $A \to \mathbb{P}^{m}$
be the  morphism  induced by the linear projection  from $z$,
where we identify $\pr^m$ with the variety of lines through $z$ in $\pr^{m+1}$.
\begin{enumerate}[1)]
\item The scheme $\textup{Hilb}^{[a]}(A/\mathbb{P}^{m})$
is connected and smooth of dimension $m$,  and the natural morphism 
$\Pi\colon\textup{Hilb}^{[a]}(A/\mathbb{P}^{m})\to \mathbb{P}^{m}$
is finite of degree $\binom{d}{a}$.
\item Let $\ell\subset \pr^{m+1}$ be a line through $z$, and let $[W]\in \textup{Hilb}^{[a]}(A/\mathbb{P}^{m})$ be a point over $[\ell]\in\pr^m$. Then $\Pi$ is smooth at $[W]$ if and only if $W$ is a union of irreducible components of $\ell\cap A$. 
\end{enumerate}
\end{thm}
Notice that the smoothness of $\textup{Hilb}^{[a]}(A/\mathbb{P}^{m})$
was proved by Gruson and Peskine
in \cite[Theorem 1.3]{grusonpeskine}, with different methods. Our proof is independent from theirs, and relies on Theorem \ref{V_x}.
\begin{parg}{\bf Outline of the paper.}
In section \ref{notations}, we introduce notations used in the remainder of the paper, and we discuss some properties of families of rational curves.

In section \ref{D} we study Fano manifolds $X$ of dimension $n\geq 3$ having a prime divisor $D$ with $\rho_D=1$, using techniques from the Minimal Model Program, and in particular  results from \cite{BCHM}. It follows from \cite[Proposition 5]{toru} that $\rho_X\leq 3$; we study the cases $\rho_X=2$ and $\rho_X=3$. 
In the case $\rho_X=2$, we describe the possible extremal contractions of $X$ (see Remark \ref{nonnefdiv} and Proposition \ref{divisorial}). Then
we give a complete classification of these varieties when $\rho_X=3$, see Example \ref{ex2} and Theorem \ref{main2}.  
This generalizes results from \cite{bonwisncamp,toru} for the case $D\cong\pr^{n-1}$ and negative normal bundle. 

In section \ref{proofs}, we specialize the results of the previous section to the case where $X$ has a locally unsplit dominating family $V$ of rational curves of anticanonical degree $n$. Indeed, for general $x\in X$, any irreducible component of the locus swept out by curves of the family through $x$ is a divisor $D$ with $\rho_D=1$. This yields $\rho_X\leq 3$.

We first prove Theorem \ref{result} on the case $l_X=n$. Next we show that, under the assumptions of Theorem \ref{main}, if moreover $\rho_X=3$, then $X$ is isomorphic to one of the varieties described in Example \ref{ex}, and we determine the  class of the family $V$ in $\N(X)$ (see Proposition \ref{prop_main}). Finally, when $\rho_X=2$, we describe the possible extremal contractions of $X$ (see Lemma \ref{contraction}).
In this section we use repeatedly the characterization of the projective space given by Theorem \ref{CMSBK}.

In section \ref{families} we first  construct a locally unsplit dominating family of rational curves on the varieties introduced in Example \ref{ex}, proving Proposition \ref{examples}. With the use of Proposition \ref{prop_main}, we also show that the family is unique, completing the proof of Theorem \ref{main}. Finally, we show Theorems \ref{V_x} and \ref{tau_x}, using 
Theorem \ref{genus_hilbert} from the Appendix (section \ref{appendix}).

In section \ref{appendix} we discuss the relative Hilbert scheme 
$\textup{Hilb}^{[a]}(A/\mathbb{P}^{m})$, where 
  $A\to\pr^m$ is the projection of a smooth hypersurface $A\subset\pr^{m+1}$ from a general point. We first study local properties of the natural morphism $\Pi\colon \textup{Hilb}^{[a]}(A/\mathbb{P}^{m})\to\pr^m$, and we show that $\textup{Hilb}^{[a]}(A/\mathbb{P}^{m})$ is integral (see Theorem \ref{genus_hilbert}). This is used in the proof of Theorem \ref{V_x}.

 In Theorem \ref{genus_hilbert} we also determine the genus of the curve $\textup{Hilb}^{[a]}(A/\mathbb{P}^{1})$; together with the description of the fibers  of $\Pi$, this is crucial for the proof of Theorem \ref{tau_x}.

At last we prove that the Hilbert scheme is smooth (Theorem \ref{app}), as a consequence of Theorem \ref{V_x}.
\end{parg}
\begin{parg}{\bf Acknowledgements.} 
This work was initiated during the conference \emph{G\'eom\'etrie Alg\'ebrique et G\'eom\'etrie Complexe} at C.I.R.M., Marseille, in March 2012; we thank the organizers for this opportunity. 
We are also grateful to Alberto Collino for some useful conversations.

The authors were partially supported by the 
   European Research Network Program GDRE-GRIFGA 
(\textit{Italian-French GDRE in Algebraic Geometry}).  The first named author was partially supported by the Research Project M.I.U.R. PRIN 2009 ``Spazi di moduli e teoria
di Lie''. The second named author was partially supported by the project \textit{CLASS} of
Agence Nationale de la Recherche, under agreement 
ANR-10-JCJC-0111.
We would like to thank these institutions for their support. 
\end{parg}
\section{Notations and preliminaries}\label{notations}
Throughout this paper, we work over the field of complex numbers.

We will use the definitions and apply the techniques of 
the Minimal Model Program frequently, without explicit references.
We refer the reader to \cite{kollarmori} and \cite{debarreUT} for 
background and details.
 
For any projective variety $X$, we denote by $\N(X)$ (respectively, $\Nu(X)$) the vector space of one-cycles  (respectively, Cartier divisors), with real coefficients, modulo numerical equivalence. We denote numerical equivalence by $\equiv$, for both one-cycles and $\Q$-Cartier divisors. We denote by $[C]$ (respectively, $[D]$) the numerical equivalence class of a curve $C$ (respectively, of a Cartier divisor $D$). Moreover, $\NE(X)\subset\N(X)$ is the convex cone generated by classes of effective curves.

\emph{For any closed subset $Z\subset X$, we denote by $\N(Z,X)$ the subspace of $\N(X)$ generated by classes of curves contained in $Z$.}

If $D$ is a Cartier divisor in $X$, we set $D^{\perp}:=\{\gamma\in\N(X)\,|\,D\cdot\gamma=0\}$.

If $X$ is a normal projective variety, a \emph{contraction} of $X$ is a surjective morphism $\ph\colon X\to Y$, with connected fibers, where $Y$ is normal and projective. The contraction is elementary if $\rho_X-\rho_Y=1$.

Let $R$ be an extremal ray of $\NE(X)$. 
If $D$ is a divisor in $X$, the sign of $D\cdot R$ is the sign of $D\cdot \Gamma$, $\Gamma$ a non-zero one-cycle with class in $R$. 

Suppose that $K_X\cdot R<0$, 
 and 
let $\ph\colon X \to Y$ be the associated elementary contraction.
We set $\Lo(R):=\Exc(\ph)$, the locus where $\ph$ is not an isomorphism.

\smallskip

By a $\pr^1$-bundle we mean a smooth morphism whose fibers are isomorphic to $\pr^1$, while a  morphism is called a conic bundle if every fiber is 
 isomorphic to a plane conic. 

If $\mathscr{E}$ is a vector bundle on a variety $Y$, we denote by $\pr_{Y}(\mathscr{E})$ the scheme $\Proj_Y(\Sym(\mathscr{E}))$.

\smallskip

We refer the reader to \cite[\S II.2 and IV.2]{kollar} for the main properties of families of rational curves; we will keep the same notation as \cite{kollar}.
In particular, we recall that $\Rat(X)$ is the normalization of the open subset of $\text{Chow}(X)$ parametrizing integral rational curves.

By a family of rational curves in $X$ we mean an irreducible component $V$ of $\Rat(X)$. We say that $V$ is a \emph{dominating family} if its universal family dominates $X$.
We say that $V$ is \emph{locally unsplit} if, for a general point $x\in X$, the subfamily of $V$ parametrizing curves through $x$ is proper.

Let $V$ be a locally unsplit dominating family of rational curves on $X$.
The class in $\N(X)$ of a curve $C$ from $V$ does not depend on the choice of $C$, and will be 
 denoted by $[V]$.

We denote by $[C]\in V$ a point corresponding to the integral rational curve $C\subset X$. We warn the reader that this is the same notation as for the numerical equivalence class of the curve $C$ in $\N(X)$. Unfortunately both notations are standard; however it will be easy for the reader to understand from the context whether we are considering the point $[C]$ in $V$ or in $\N(X)$.

For $x\in X$, we denote by $V_x$ the \emph{normalization} of the closed subset of $V$ parametrizing curves through the point $x$, and by $\Lo(V_x)\subseteq X$ 
the  union of all curves of the family $V_x$.

We denote by $\Hom(\pr^1,X,0\mapsto x)$ the scheme parametrizing morphisms from $\pr^1$ to $X$ sending $0\in\pr^1$ to $x$. 

We now recall some well-known properties.

Suppose that $x\in X$ is \emph{general}.
Then every curve in $V_x$ is free. This implies that $V_x$ is smooth, of dimension
$-K_X\cdot [V]-2$, but possibly not connected. Moreover there is a smooth closed subset $\wi{V}_x\subset\Hom(\pr^1,X,0\mapsto x)$, which is a union of irreducible components, containing all birational maps $f\colon\pr^1\to X$ such that $f(0)=x$ and $f(\pr^1)$ is a curve of the family $V_x$. There is an induced smooth morphism $\widehat{V}_x\to V_x$, sending $[f]$ to $[f(\pr^1)]$. 

Still for a general point $x$, if a curve from $V_x$ 
is smooth at $x$, then it is parametrized by a unique point of $V_x$.

We say that an integral rational curve $C\subset X$ is \emph{standard} if the pull-back of ${T_X}_{|C}$ under the normalization $\pr^1\to C$ is $\mathscr{O}_{\pr^1}(2)\oplus\mathscr{O}_{\pr^1}(1)^p\oplus\mathscr{O}_{\pr^1}^{n-1-p}$ for some $p\in\{0,\dotsc,n-1\}$.
\section{Fano manifolds containing a prime divisor with Picard number one}\label{D}
In this section we study Fano manifolds $X$ having a 
prime divisor $D$ with $\rho_D=1$, or more generally $\dim\N(D,X)=1$. 
The main technique here is the study of extremal rays and contractions of $X$. 

The first step for the proofs of Theorems \ref{result} and \ref{main2} is the following Lemma. It is a standard application of Mori theory, in particular
the proof can be adapted from  \cite[proof of Proposition 5]{toru}, and follows the same
strategy used in \cite{bonwisncamp}. 
We give a short proof for the reader's convenience.
\begin{lemma}\label{uno}
Let $X$ be a Fano manifold of dimension $n\geq 3$ and Picard number $\rho_X>1$, and let $D\subset X$ be a prime divisor with $\dim\N(D,X)=1$. Then one of the following holds:
\begin{enumerate}[$(i)$]
\item $\rho_X=2$ and there exists a blow-up
$\sigma\colon X\to Y$ with center $A_Y\subset Y$ smooth of codimension $2$, $Y$ is smooth and Fano, and $D\cdot R>0$, where $R$ is the extremal ray of $\NE(X)$ generated by the class of a curve contracted by  $\sigma$;
\item $\rho_X=2$ and there exists a conic bundle $\sigma\colon X\to Y$, finite on $D$, such that $Y$ is smooth and Fano;
\item $\rho_X=3$ and there is a conic bundle $\ph\colon X\to Z$, finite on $D$, such that $Z$ is smooth, Fano, and $\rho_Z=1$.
The conic bundle $\ph$ is the contraction of a face $R+\widehat{R}$ of $\NE(X)$, where $R$ and $\wi{R}$ both correspond to a smooth blow-up of a codimension $2$ subvariety. Moreover $D\cdot R>0$, and we have a diagram:
$$\xymatrix{&X\ar[dl]_{\wi{\sigma}}\ar[dd]^{\ph}\ar[dr]^{\sigma}&\\
{\wi{Y}}\ar[dr]_{\wi{\pi}}&&Y\ar[dl]^{\pi}\\
&Z&
}$$
where $\sigma$ is the contraction of $R$, $\wi{\sigma}$ is the contraction of $\wi{R}$, $Y$ and $\wi{Y}$ are smooth, $Y$ is Fano, the center $A_Y\subset Y$ of the blow-up $\sigma$ is contained in $D_Y:=\sigma(D)$, and $\pi$ and $\wi{\pi}$ are conic bundles.
\end{enumerate}
\end{lemma}
\begin{proof}
Let $R$ be an extremal ray of $\NE(X)$ with $D\cdot R>0$, and
 $\sigma\colon X\to Y$ the associated contraction.

If $R\subset\N(D,X)$, then every curve contained in $D$ has class
in $R$, hence $\sigma(D)$ is a point and $D\subseteq\Lo(R)$. We conclude that $\Lo(R)=X$, because $D\cdot R>0$. On the other hand, since $\rho_X>1$, we can find a non-trivial fiber $F$ of $\sigma$ disjoint from $D$, which yields $D\cdot R=0$, a contradiction. Therefore $R\not\subset\N(D,X)$.

This implies that $\sigma$ is finite on $D$, therefore every non-trivial fiber of $\sigma$ has dimension one. Thus $Y$ is smooth and there are two possibilities: either $\sigma$ is a conic bundle, or it is the blow-up of $A_Y\subset Y$ with $A_Y$ smooth of codimension $2$ (see \cite[Theorem 1.2]{wisn}).

If  $\sigma$ is a conic bundle, then $\rho_Y=1$ because $\sigma(D)=Y$, so 
we are in case $(ii)$.
If $\sigma$ is a blow-up and 
 $\rho_Y=1$, then we are in case $(i)$. 

Assume that $\sigma$ is a blow-up and $\rho_Y\geq 2$, and set $D_Y:=\sigma(D)$.  Then $D_Y$ is a prime divisor in $Y$ and
  $\N(D_Y,Y)=\sigma_*(\N(D,X))$, hence 
$\dim\N(D_Y,Y)=1$. 
Moreover $A_Y\subset D_Y$. 

Let $E\subset X$ be the exceptional divisor. We have 
$-K_X+E=\sigma^*(-K_Y)$.
If $C\subset A_Y$ is an irreducible curve, and $C'\subset D_Y$ is an irreducible curve not contained in $A_Y$,
then there exists $\lambda\in\Q_{>0}$ such that $C\equiv\lambda C'$, 
so that $-K_Y\cdot C=\lambda (-K_Y\cdot C')
=\lambda(-K_X\cdot \widetilde{C})+\lambda(E\cdot \widetilde{C})>0$, 
where $\widetilde{C}$ is the strict transform of $C'$ in $X$. This implies that
$Y$ is Fano (see \cite{wisn}).

We repeat the same argument in $Y$ and
 take an extremal ray $R_2\subset\NE(Y)$ with $D_Y\cdot R_2>0$. 

Similarly as before we see that $R_2\not\subset\N(D_Y,Y)$, so that again, if 
 $\pi\colon Y\to Z$ is  the contraction of $R_2$,  $\pi$ is finite on $D_Y$ and has fibers of dimension at most one.
  Hence as before $\pi$ is either a conic bundle or  the smooth blow-up of a subvariety of codimension $2$ in $Z$.

If $\pi$ is a conic bundle, then $\rho_Z=1$ because $\pi(D_Y)=Z$, so $\rho_X=3$. 
Set $\ph:=\pi\circ\sigma\colon X\to Z$; then $\ph$ has a second factorization 
$X\stackrel{\wi{\sigma}}{\to}\wi{Y}\stackrel{\wi{\pi}}{\to} Z$.
 Since every fiber of $\ph$ has dimension one, 
 both $\wi{\sigma}$ and $\wi{\pi}$ have fibers of 
 dimension $\le 1$. Applying \cite[Theorem 4.1]{AWaview} we conclude that $\ph$ and  $\wi{\pi}$ are conic bundles, $\wi{Y}$ is smooth, and $\wi{\sigma}$ the blow-up of a smooth subvariety of codimension $2$, so 
 we are in case $(iii)$.


Finally,  $\pi$ cannot be a blow-up. Indeed if so, $\Exc(\pi)$ is a prime divisor which intersects $D_Y$, and since $\dim\N(D_Y,Y)=1$, $\Exc(\pi)$ has strictly positive intersection with every curve contained in $D_Y$. In particular $\Exc(\pi)$ must intersect $A_Y$, as $\dim A_Y=n-2\geq 1$. If $F$ is a non-trivial fiber of $\pi$ with $F\cap A_Y\neq\emptyset$, and $\widetilde{F}\subset X$ is its strict transform,
one has $-K_X\cdot\widetilde{F}<-K_Y\cdot F=1$, 
 a contradiction.
\end{proof}
Remark \ref{nonnefdiv} and Proposition \ref{divisorial} below   describe the possible extremal contractions of $X$ in the case $\rho_X=2$.
\begin{remark}\label{nonnefdiv}
Let $X$ be a Fano manifold, and $D\subset X$ a prime divisor with $\dim\N(D,X)=1$. If $D$ is not nef, then there exists a unique extremal ray $R$ such that $D\cdot R<0$; the contraction associated to $R$ is divisorial and sends $D$ to a point.

Indeed, let $R$ be an extremal ray of $\NE(X)$ such that 
$D \cdot R<0$, and $\sigma$ the associated contraction. Notice that 
$\Exc(\sigma)\subseteq D$ since $D \cdot R<0$. On the other hand,
every curve contained in $D$ has class in $R$
since $\dim\N(D,X)=1$. This implies that
$D = \Exc(\sigma)$, and that $\sigma(D)$ is a point.
\end{remark}
\begin{proposition}\label{divisorial}
Let $X$ be a Fano manifold of dimension $n\geq 3$ and Picard number $\rho_X=2$, and let $D\subset X$ be a nef prime divisor with $\dim\N(D,X)=1$. Then $S:=D^{\perp}\cap\NE(X)$ is an extremal ray of $X$, and one of the following holds:
\begin{enumerate}[$(i)$]
\item the contraction of $S$ is a fiber type contraction onto $\pr^1$, having $D$ as a fiber;
\item the contraction of $S$ is  divisorial, sends its exceptional divisor $G$ to a point, and $G\cap D=\emptyset$;
\item the contraction of $S$ is small, it has a flip $X\dasharrow X'$, $X'$ is smooth, and there is a $\pr^1$-bundle $\psi\colon X'\to Y'$. Moreover $\psi$ is finite on the strict transform of $D$ in $X'$.
\end{enumerate}
Furthermore, if there exists a smooth, irreducible subvariety $A\subset D$, of codimension $2$, such that the blow-up of $X$ along $A$ is Fano, then $(iii)$ cannot happen.
\end{proposition}
\noindent\emph{Proof.}
\begin{pargtwo}
We first notice that $D$ is not ample, because  $\N(D,X)\subsetneq\N(X)$. Indeed  the push-forward of one-cycles $\N(D)\to\N(X)$ is not surjective, so that 
the restriction map $\Nu(X)\to\Nu(D)$ is not injective. We have $\Nu(X)\cong H^2(X,\R)$ (because $X$ is Fano) and $\Nu(D)\hookrightarrow H^2(D,\R)$, hence the restriction map $H^2(X,\R)\to H^2(D,\R)$ is not injective as well. By the Lefschetz Theorem on hyperplane sections, $D$ cannot be ample (recall that $n\geq 3$).

Since $D$ is nef and non-ample, and $\rho_X=2$, 
we conclude that $D^{\perp}\cap\NE(X)$ is an extremal ray $S$ of $\NE(X)$. Set $G:=\Lo(S)\subseteq X$.
\end{pargtwo}
\begin{pargtwo}
If $S\subset\N(D,X)$, then the contraction of $S$ sends $D$ to a point, and hence $D\subseteq G$. On the other hand $D\cdot S=0$, thus 
$D$ is the pull-back of a Cartier divisor. Therefore the target of the contraction of $S$ is $\pr^1$, $D$ is a fiber, and we are in case $(i)$.
\end{pargtwo}
\begin{pargtwo}
We assume that $S\not\subset\N(D,X)$.
 If $G\cap D\neq\emptyset$, then $D$ must intersect some irreducible 
curve $C$ with class in $S$, and this yields $C\subseteq D$, because $D\cdot C=0$. This contradicts  $S\not\subset\N(D,X)$, therefore  $G\cap D=\emptyset$. In particular, 
 the contraction of $S$ is birational. Finally $\N(G,X)\subseteq D^{\perp}$ has dimension one; this implies that the contraction of $S$ maps $G$ to points.
If  $S$ is a divisorial extremal ray, then we are in case $(ii)$.
\end{pargtwo}
\begin{pargtwo}
Suppose now that the contraction of $S$ is small; by \cite[Corollary 1.4.1]{BCHM} the flip 
$X\dasharrow X'$
of $S$ exists. 
Let $D'\subset X'$ be the strict transform of $D$, $S'$ the small extremal ray of $X'$ associated with the flip, and $G':=\Lo(S')$.

Notice that $G'\cap D'=\emptyset$, as $G\cap D=\emptyset$. 

Notice also that $X'$ has normal, $\Q$-factorial, and terminal singularities, and 
$\Sing(X')\subseteq G'$. We have $\rho_{X'}=\rho_X=2$ and $K_{X'}\cdot S'>0$; on the other hand, a curve disjoint from $G'$ has positive anticanonical degree. In particular, 
by the Cone Theorem,
$X'$ has a second extremal ray $T$ with $-K_{X'}\cdot T>0$, and 
$$\NE(X')=\R_{\geq 0}S'+\R_{\geq 0}T.$$

Since the flip $X\dasharrow X'$ is an isomorphism in a neighborhood of $D$, the linear subspace 
$\N(D',X')\cong \N(D,X)$ stays one-dimensional.
 Moreover $D'\cdot S'=0$, hence we must have
 $D'\cdot T>0$.

Let $\psi\colon X'\to Y'$ be the contraction of $T$.
Arguing as in the proof of Lemma \ref{uno}, we see that $T\not\subset\N(D',X')$.
Since the contraction of $S$ sends $G$ to points, the  
contraction of $S'$ sends $G'$ to points. This implies
that every curve in $G'$ has class in the extremal ray $S'$.

 We deduce that $\psi$ is finite on both $D'$ and $G'$.

In particular, since  $D'\cdot T>0$,
 every non-trivial fiber of $\psi$ has dimension one. 
\end{pargtwo}
\begin{pargtwo}\label{degrees}
Let $C$ be an irreducible component of a non-trivial fiber of $\psi$. If $C\cap G'\neq\emptyset$, then:
$$
-K_{X'}\cdot C>1.
$$
Indeed if $\widetilde{C}\subset X$ is the strict transform of $C$, we have $-K_{X'}\cdot C> -K_X\cdot\widetilde{C}\geq 1$. This follows from \cite[Lemma 3.38]{kollarmori}, see \cite[Lemma 3.8]{31} for an explicit proof.
\end{pargtwo}
\begin{pargtwo} We show that $\psi$ is of fiber type.
By contradiction, 
 assume that $\psi$ is birational.

Suppose that $\Exc(\psi)\cap G'\neq\emptyset$, and let $F_0$ be  an irreducible component of a fiber of $\psi$ which intersects $G'$. Notice that $F_0\not\subseteq G'$ (since $\psi$ is finite on $G'$), 
in particular $F_0\not\subseteq \Sing(X')$. We get $-K_{X'}\cdot {F}_0\leq 1$ by \cite[Lemma 1.1]{ishii}, and
$-K_{X'}\cdot {F}_0>1$ 
by \ref{degrees}, a contradiction.

Therefore $\Exc(\psi)\cap G'=\emptyset$, so that $\Exc(\psi)$ is contained in the smooth locus of $X'$. By \cite[Theorem 1.2]{wisn}, $\Exc(\psi)$ is a divisor. We have
$\Exc(\psi)\cdot S'=0$ and $\Exc(\psi)\cdot T<0$, hence $-\Exc(\psi)$ is nef, a contradiction.

Thus $\psi$ is of fiber type. 
\end{pargtwo}
\begin{pargtwo}
We show that $\psi \colon X'\to Y'$ is a $\mathbb{P}^1$-bundle with $X'$ and $Y'$ smooth, so that we are in case $(iii)$.

Since $\Sing(X')$ cannot dominate $Y'$, the general fiber of $\psi$ is a smooth rational curve of anticanonical degree $2$. 

Suppose that there is a fiber $F$ of $\psi$ such that
the corresponding one-cycle is not integral
and $G'\cap F\neq\emptyset$.
Then there is an irreducible component $C$ of $F$, such that 
$-K_{X'}\cdot C \le 1$. 
If $C\cap G'=\emptyset$, then $-K_{X'}$ is Cartier in a neighbourhood of $C$, and we must have 
$-K_{X'}\cdot C = 1$ and $-K_{X'}\cdot (F-C) = 1$ (where we consider $F$ as a one-cycle). 
Thus, up to replacing $C$ with another irreducible component of 
$F$,
we may assume that $C\cap G'\neq\emptyset$, and 
$-K_{X'}\cdot C \le 1$; but this contradicts \ref{degrees}.

By \cite[Theorem II.2.8]{kollar}, 
$\psi$ is smooth in a neighbourhood of $\psi^{-1}(\psi(G'))$. Thus
$\Sing(X')=\psi^{-1}(\Sing(Y')\cap\psi(G'))$, because
 $\Sing(X')\subseteq G'$. This implies that 
$\Sing(X')=\emptyset$, since $\psi$ is finite on $G'$. In particular, $Y'$ is smooth (see \cite[Theorem 4.1(2)]{AWaview} and references therein), $\psi$ is a conic bundle,
and either
the discriminant locus $\Delta$ of $\psi$ has pure codimension one or $\Delta=\emptyset$. If $\Delta\neq\emptyset$, then $\Delta$ is an ample divisor in $Y'$, because 
$\rho_{Y'}=1$. Hence
$\psi(G')\cap \Delta\neq\emptyset$ ($\psi$ is finite on $G'$ and $\dim G'\ge 1$), a contradiction.
This proves that $\psi \colon X'\to Y'$ is a $\mathbb{P}^1$-bundle with $X'$ and $Y'$ smooth.
\end{pargtwo}
\begin{pargtwo}
Suppose now that we are in case $(iii)$, and that
there is a smooth irreducible subvariety $A\subset D$ of codimension $2$, such that the blow-up of $X$ along $A$ is Fano. We show that this gives a contradiction.

Let 
 $A'\subset {D'}$ the strict transform of $A$, and
 let us consider the divisor $\psi^*(\psi(A'))$ in $X'$. Since $\psi^*(\psi(A'))\cdot T=0$, we must have $\psi^*(\psi(A'))\cdot S'>0$.
Therefore we find a fiber $F'$ of $\psi$ which intersects both $A'$ and $G'$. 

Let $F\subset X$ be the strict transform of $F'$. As in  \ref{degrees}, we see that $-K_X\cdot F< -K_{X'}\cdot F'=2$, so that 
$-K_X\cdot F=1$. On the other hand $F\cap A\neq\emptyset$ and $F\not\subseteq A$, hence the strict transform of $F$ in the blow-up of $X$ along $A$ should have non-positive anticanonical degree, a contradiction.
\qed
\end{pargtwo}

\smallskip

We will show that 
any Fano manifold $X$ 
with $\rho_X=3$
having a 
prime divisor $D$ with $\dim\N(D,X)=1$
is isomorphic to one of the varieties described below.

We first recall the following definition for the reader's convenience.
Let $\mathscr{L}$ be an ample line bundle on a normal projective variety
$Z$. 
Consider the $\pr^1$-bundle  $Y=\pr_Z(\mathscr{O}_Z\oplus\mathscr{L})$, with natural projection
$\pi \colon Y\to Z$. The tautological line bundle $\mathscr{O}_Y(1)$ is semiample on $Y$.
For $m\gg 0$, the linear system $|\mathscr{O}_Y(m)|$ induces a birational morphism $Y \to Y_0$ onto a normal projective variety, contracting 
the divisor $E=\pr_Z(\mathscr{O}_Z)\cong Z\subset Y$ corresponding to the projection 
$\mathscr{O}_Z\oplus\mathscr{L}\twoheadrightarrow \mathscr{O}_Z$ to a point.
Following \cite{beltrametti_sommese}, we call $Y_0$ the \emph{normal generalized cone over the base $(Z,\mathscr{L})$}.
\begin{example}\label{ex2}
Fix integers $n$, $a$, and $d$, such that $n\ge 3$, $a\ge 0$, and $d\geq 1$.

Let $Z$ be a Fano manifold of dimension $n-1$, with $\rho_Z=1$.
Let $\mathscr{O}_Z(1)$ be the ample generator of $\textup{Pic}(Z)$.
If $m$ is an integer, then we write $\mathscr{O}_Z(m)$ for $\mathscr{O}_Z(1)^{\otimes m}$.
Let moreover $A\in|\mathscr{O}_Z(d)|$ be a smooth hypersurface.

Set $Y:=\pr_{Z}(\ol_{Z}\oplus\ol_{Z}(a))$, and let $\pi\colon Y\to Z$ be the $\pr^1$-bundle. 

If $a>0$,
then there is a birational contraction $Y\to Y_0$ sending a divisor $G_Y$ to a point, where $Y_0$ is the normal generalized cone over $(Z,\mathscr{O}_Z(a))$. We have $G_Y\cong Z$, $G_Y$ is a section of $\pi$, and $\mathcal{N}_{G_Y/Y}\cong\ol_{Z}(-a)$.

If $a=0$, then $Y\cong Z\times\pr^1$. Let $G_Y$ be a fiber of 
$Y\to \pr^1$. We have $G_Y\cong Z$, $G_Y$ is a section of $\pi$, and $\mathcal{N}_{G_Y/Y}\cong\ol_{Z}$.

Let now $\wi{G}_Y\cong Z\subset Y$ be a section of $\pi$ with normal bundle $\mathcal{N}_{\wi{G}_Y/Y}\cong\ol_{Z}(a)$.
Notice that $G_Y\cap \wi{G}_Y=\emptyset$ if $a>0$. If $a=0$, we choose $\wi{G}_Y$ such that $G_Y\cap \wi{G}_Y=\emptyset$.
Set $A_Y:=\wi{G}_Y\cap\pi^{-1}(A)$.
Finally let $\sigma\colon X\to Y$ be the blow-up of $A_Y$ (see Figure \ref{fig1}).

Then $X$ is a smooth projective variety of dimension $n$, with $\rho_X=3$. 

\stepcounter{remarktwo}
\begin{figure}\begin{center} 
\scalebox{0.4}{\includegraphics{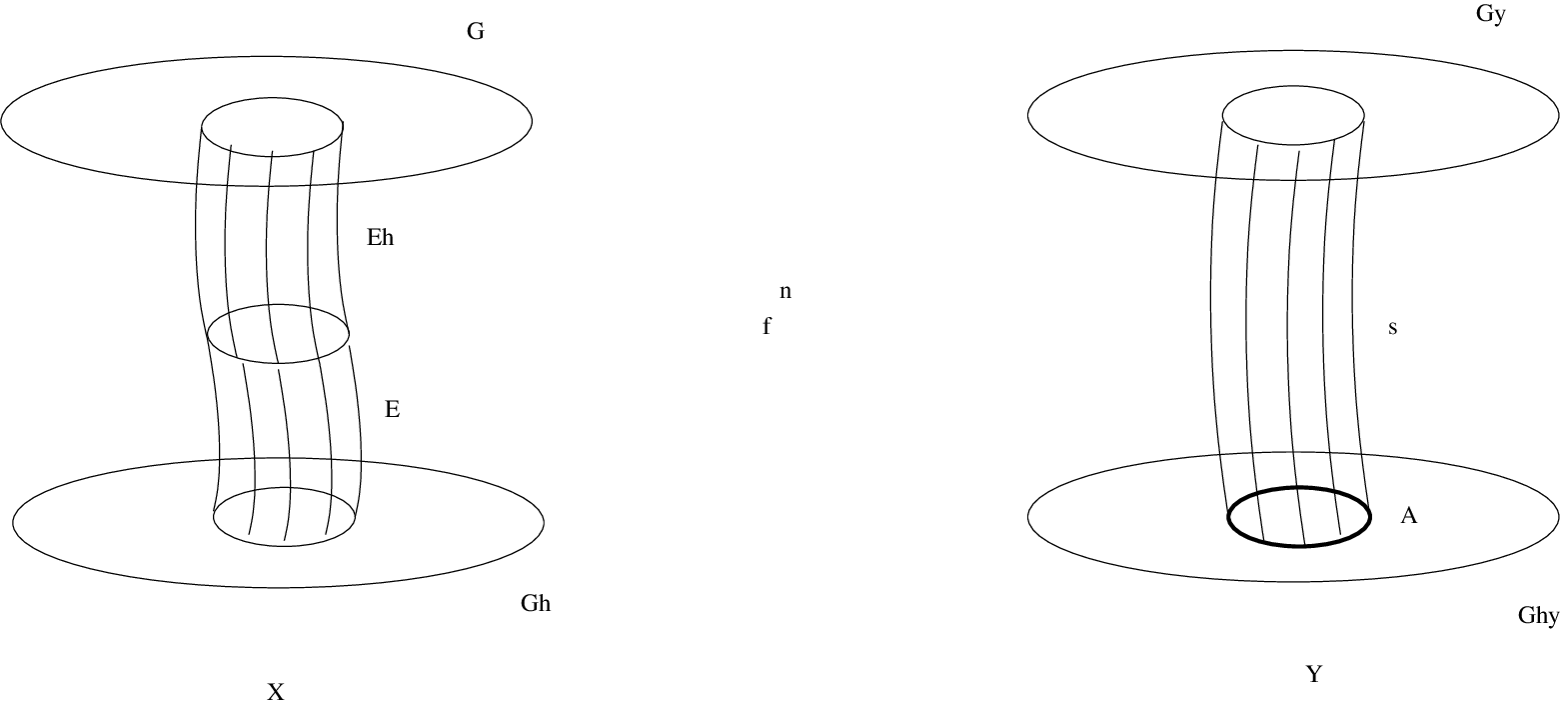}}
\caption{The blow-up $\sigma$.}\label{fig1}
\end{center}\end{figure}

Let $G,\wi{G}\subset X$ be the tranforms of $G_Y,\wi{G}_Y\subset Y$, respectively. Then $G\cong \wi{G}\cong Z$, $\mathcal{N}_{G/X}\cong\ol_{Z}(-a)$, and $\mathcal{N}_{\wi{G}/X}\cong\ol_{Z}(-(d-a))$.

The composition $\ph:=\pi\circ\sigma\colon X\to Z$ is a conic bundle, and has a second factorization:
 $$\xymatrix{ &X\ar[dl]_{\wi{\sigma}}\ar[dr]^{\sigma}\ar[dd]^{\ph}  &\\
 {\wi{Y}}\ar[dr]_{\wi{\pi}}&& Y\ar[dl]^{\pi}\\
 & Z &
 }$$
 where $\wi{Y}=\pr_{Z}(\ol_{Z}\oplus\ol_{Z}(d-a))$.
The images $\wi{\sigma}(G)$ and $\wi{\sigma}(\wi{G})$ are disjoint sections of the $\pr^1$-bundle $\wi{\pi}\colon\wi{Y}\to Z$, with normal bundles $\mathcal{N}_{\wi{\sigma}(G)/\wi{Y}}=\ol_{Z}(d-a)$ and $\mathcal{N}_{\wi{\sigma}(\wi{G})/\wi{Y}}=\ol_{Z}(a-d)$. Moreover $\wi{\sigma}$ is the blow-up of $\wi{Y}$ along the intersection
$\wi{\sigma}(G)\cap\wi{\pi}^{-1}(A)$.

Set $E:=\Exc(\sigma)$ and $\wi{E}:=\Exc(\wi{\sigma})$, and
let $F\subset E$ and $\wi{F}\subset\wi{E}$ be exceptional fibers of $\sigma$ and $\wi{\sigma}$, respectively.
Let moreover $C_Z\subset Z$ be an irreducible curve having minimal intersection with $\mathscr{O}_Z(1)$, and set $\delta:=\mathscr{O}_Z(1)\cdot C_Z$. Finally 
 let $C_G\subset G$ and $C_{\wi{G}}\subset \wi{G}$ be curves corresponding to $C_Z$. We have the following relations of numerical equivalence:
\renewcommand{\theequation}{\theremarktwo}
\stepcounter{remarktwo}
\begin{equation}\label{equiv}
C_G+a\delta\wi{F}\equiv C_{\wi{G}}+(d-a)\delta F,\qquad dG+a\wi{E}\equiv d\wi{G}+(d-a)E,
\end{equation}
and the relevant intersections are shown in table \ref{table}, where $i_Z$ is the index of $Z$, \textit{i.e.} the integer defined by $\mathscr{O}_Z(-K_Z)\cong\mathscr{O}_Z(i_Z)$.
\stepcounter{remarktwo}
\begin{table}\caption{Intersection table in $X$.}\label{table}
$$\begin{array}{|r|r|r|c|c|c|}
\hline
&  F& \wi{F}& C_G& C_{\wi{G}}&C_G+a\delta\wi{F}\\\hline
E&-1&1&0&d\delta&a\delta\\\hline
\wi{E}&1&-1&d\delta&0&(d-a)\delta\\\hline
G&0&1&-a\delta&0&0\\\hline
\wi{G}&1&0&0&-(d-a)\delta&0\\\hline
-K_X&1&1&(i_Z-a)\delta&(i_Z-(d-a))\delta&i_Z\delta\\ \hline
\end{array}
$$
\end{table}
\begin{lemma}\label{NE(X)}
The cone $\NE(X)$ is closed and polyhedral.

If $a=0$, then $\NE(X)$ has three extremal rays, generated by the classes of $F$, $\wi{F}$, and $C_{\wi{G}}$, with loci $E$, $\wi{E}$, and $\wi{G}$ respectively.

If $a\ge d$, then $\NE(X)$ has three extremal rays, generated by the classes of $F$, $\wi{F}$, and $C_G$, with loci $E$, $\wi{E}$, and $G$ respectively.

If instead $0<a<d$, then $\NE(X)$ is non-simplicial and 
has $4$ extremal rays, generated by the classes of  $F$, $\wi{F}$, $C_G$, and $C_{\wi{G}}$, with loci $E$, $\wi{E}$, $G$ and $\wi{G}$ respectively. 
\end{lemma}
\begin{proof}
Set $R:=\R_{\geq 0} [F]$, $\wi{R}:=\R_{\geq 0} [\wi{F}]$, $S:=\R_{\geq 0} [C_G]$, and $\wi{S}:=\R_{\geq 0} [C_{\wi{G}}]$. We already know that $R$ and $\wi{R}$ are extremal rays of $\overline{\NE}(X)$, and that $R+\wi{R}$ is a face.

Since $\mathscr{E}:=\ol_{Z}\oplus\ol_{Z}(a)$ is nef and non-ample,  $\ol_Y(\wi{G}_Y)=\ol_{\pr_Z(\mathscr{E})}(1)$ is nef and non-ample in $Y$, and the same holds for $\sigma^*(\wi{G}_Y)=\wi{G}+E$ in $X$. It is not difficult to see that
$(\wi{G}+E)^{\perp}\cap \overline{\NE}(X)=R+S$. 
In particular, this shows that $S$ is an extremal ray of $\overline{\NE}(X)$, so that there exists a nef divisor $H$ such that $H^{\perp}\cap \overline{\NE}(X)=S$. 

If $0 < a < d$, then similarly as before 
$\wi{\sigma}(G)$ is nef and non-ample in $\wi{Y}$, so that $\wi{\sigma}^*(\wi{\sigma}(G))=G+\wi{E}$ is nef in $X$. The plane
 $(G+\wi{E})^{\perp}$ intersects $\overline{\NE}(X)$ along the face
$\wi{R}+\wi{S}$; in particular, $\wi{S}$ is an extremal ray.

Finally, the divisor
$$D:=\left(-\wi{G}\cdot C_{\wi{G}}\right)H + \left(H\cdot C_{\wi{G}}\right)\wi{G}$$
is nef, it is not numerically trivial, and $D\cdot S=D\cdot\wi{S}=0$.
Therefore $D^{\perp}\cap\overline{\NE}(X)=
S+\wi{S}$, and we get the statement in the case $0 < a < d$.

If instead $a \ge d$, then
$\wi{\sigma}(\wi{G})$ is nef, non-ample in $\wi{Y}$, so that
$\wi{G} = \wi{\sigma}^*(\wi{\sigma}(\wi{G}))$ is nef, and does not intersect $G$. We get $\wi{G}^{\perp}\cap\overline{\NE}(X)=
\wi{R}+S$, which gives the statement in the case $a \ge d$.

The case $a=0$ follows from the case $a=d$, see Remark \ref{symm}.
\end{proof}
A straightforward consequence of Lemma \ref{NE(X)} is the following.
\begin{remark}
$X$ is Fano if and only if $a\leq i_Z-1$ and $d-a\leq i_Z-1$.

Indeed, since $\NE(X)$ is closed and polyhedral, $X$ is Fano if and only if 
 every extremal ray of $\NE(X)$ has positive intersection with the anticanonical divisor (see Table \ref{table}).
\end{remark}
\begin{remark}\label{symm}
Suppose that $a\leq d$. Then by choosing $a'=d-a$, we get a variety $X'$ isomorphic to $X$, with the roles of $Y$ and $\wi{Y}$ interchanged.
\end{remark}
Example \ref{ex} is a special case of this example, with $Z=\pr^{n-1}$, and the additional condition $a\leq d$.
\end{example}
We are now in position to prove the main result of this section; see \cite{kentofujita} for a related result.
\begin{thm}\label{main2}
Let $X$ be a Fano manifold of dimension $n\geq 3$ and 
Picard number $\rho_X=3$, and let $D\subset X$ be a prime divisor with $\dim\N(D,X)=1$. 
Then $X$ is isomorphic to one of the varieties described in Example \ref{ex2}.
\end{thm}
Note that if $X$ is as in Example \ref{ex2}, then $G$ is a prime divisor with $\rho_G=1$, and hence $\dim\N(G,X)=1$. 

\bigskip

\noindent\emph{Proof of Theorem \ref{main2}.}
\begin{pargtwo}\label{primo}
As $\rho_X=3$, we are in case $(iii)$
of Lemma \ref{uno}, and there is a conic bundle $\ph\colon X\to Z$, finite on $D$. We keep the same notation as in Lemma \ref{uno}; in particular we recall the diagram: 
$$\xymatrix{&X\ar[dl]_{\wi{\sigma}}\ar[dd]^{\ph}\ar[dr]^{\sigma}&\\
{\wi{Y}}\ar[dr]_{\wi{\pi}}&&Y\ar[dl]^{\pi}\\
&Z&
}$$
We set $E:=\Lo(R)$ and $\wi{E}:=\Lo(\wi{R})$.
Notice that $D\neq\wi{E}$ because $\ph$ is finite on $D$, hence $D\cdot\wi{R}\geq 0$. Thus we may have $D\cdot\wi{R}>0$ (if $D\cap\wi{E}\neq\emptyset$) or $D\cdot\wi{R}=0$ (if $D\cap\wi{E}=\emptyset$). In the first case $\sigma\colon X\to Y$ and $\wi{\sigma}\colon X\to\wi{Y}$ have the same properties with respect to $X$ and $D$, so that their role is interchangeable, while in the second case the behaviour of the two blow-ups with respect to $D$ is different.
\end{pargtwo}
\begin{pargtwo}\label{int}
 We show that every prime divisor $B\subset X$ must intersect $E\cup\wi{E}$. 

We first notice that $\sigma(B)\cap \sigma(\wi{E})\neq\emptyset$ in $Y$.
Indeed, if $\pi(\sigma(B))=Z$, then the claim is obvious. Otherwise, $\sigma(B)=\pi^{-1}(\ph(B))$, and the claim follows from 
$\rho_Z=1$. Thus, if $A_Y\cap\sigma(B)\neq\emptyset$, then
$B$ intersects $E$. Otherwise, $B$ intersects $\wi{E}$.
\end{pargtwo}
\begin{pargtwo}\label{smooth}
We show that $\pi\colon Y\to Z$ is a smooth morphism. Otherwise it has a non-empty  discriminant divisor $\Delta\subset Y$, whose inverse image $\sigma^{-1}(\Delta)$ must be disjoint from $E\cup\wi{E}$ since $X$ is Fano, which contradicts \ref{int}.
\end{pargtwo}
\begin{pargtwo}\label{DYnef}
Let us consider the prime divisor $D_Y=\sigma(D)\subset Y$. We have $\N(D_Y,Y)=\sigma_*(\N(D,X))$, hence 
$\dim\N(D_Y,Y)=1$. 

We show that up to replacing $D$ with another prime divisor $D'\subset X$, we can assume that $D_Y$ is nef.

Suppose that $D_Y$ is not nef. Then $D_Y$
is the exceptional locus of a divisorial contraction 
$Y \to Y_0$
sending $D_Y$ to a point, by Remark \ref{nonnefdiv}.
The exceptional locus of the composite map $X \to Y_0$ is $D\cup E$, and is contracted to a point in $Y_0$. In particular, $D$ is not nef either (see for instance \cite[Lemma 3.39]{kollarmori}).

Let $D'\subset X$ be the pull-back of a general prime divisor in $Y_0$. Then $D'$ is a prime nef divisor in $X$, and
 $D'\cap(D\cup E)=\emptyset$. Notice that the classes $[D]$ and $[E]$ in $\Nu(X)$ cannot be proportional, therefore the planes $D^{\perp}$ and $E^{\perp}$ in $\N(X)$ are distinct. As $\N(D',X)\subseteq  D^{\perp}\cap E^{\perp}$, we get  $\dim\N(D',X)=1$.

By \ref{int}, we have $D'\cap\wi{E}\neq\emptyset$, hence $D'\cdot \wi{R}>0$. Moreover, as $D'$ is nef,
 we must have $\wi{\sigma}(D')$ nef too. 
This shows that
up to replacing $D$ with $D'$, and $R$ with $\wi{R}$, we can assume that $D_Y$ is nef.
\end{pargtwo}
\begin{pargtwo}\label{applydivisorial}
By \ref{DYnef}, we can assume that $D_Y$ is nef. Then 
Proposition  \ref{divisorial} applies,  and $D_Y^{\perp}\cap\NE(Y)$ is an extremal ray $S_Y$ of $\NE(Y)$.  Moreover, since $A_Y\subset D_Y$ and the blow-up of $Y$ along $A_Y$ is Fano, case $(iii)$ of Proposition \ref{divisorial} is excluded.

Suppose that we are in case $(i)$ of Proposition \ref{divisorial}. Then $Y\cong Z\times\pr^1$ (see for instance \cite[Lemma 4.9]{31}), and $D_Y$ is a fiber of 
$Y \to \pr^1$. Since $A_Y\subset D_Y$, $X$ is isomorphic to one of the varieties described in Example \ref{ex2}, with $a=0$.
This completes the proof of Theorem \ref{main2} in this case.
\end{pargtwo}
\begin{pargtwo} We
assume that we are in case $(ii)$ of   Proposition \ref{divisorial}, so that
the extremal ray $S_Y$ is divisorial, its contraction sends the divisor $G_Y:=\Lo(S_Y)$ to a point, and $D_Y\cap G_Y=\emptyset$. In particular, $\dim\N(G_Y,Y)=1$.

Denote by $\ol_Z(1)$ the ample generator of $\Pic(Z)$. 
By Lemma \ref{section_bundle} below, $G_Y\subset Y$ is a section of $\pi\colon Y\to Z$, and there exists an integer $a>0$ such that
$\mathcal{N}_{G_Y/Y}\cong \ol_Z(-a)$
and $Y\cong\mathbb{P}_Z(\ol_Z\oplus\ol_Z(a))$.
\end{pargtwo} 
\begin{pargtwo}\label{secondsectionY}
Suppose that there exists a section $\wi{G}_Y\subset Y$ of $\pi\colon Y\to Z$, disjoint from $G_Y$, and containing $A_Y$.
We claim that this implies the statement. 
Indeed  we have
$$\ol_Y(\wi{G}_Y)=\ol_Y(G_Y)\otimes\pi^*(\ol_Z(r))$$ for some $r\in\mathbb{Z}$, and since $G_Y\cap\wi{G}_Y=\emptyset$, restricting to $G_Y$ we get $r=a$, and restricting to $\wi{G}_Y$ we get $\mathcal{N}_{\wi{G}_Y/Y}\cong\ol_Z(a)$. Thus 
 $X$ is  one of the varieties described in Example \ref{ex2}, for $a>0$.
\end{pargtwo}
\begin{pargtwo}
Let $G\subset X$ be the strict  transform of  $G_Y$. We have
$G\cap (D\cup E)=\emptyset$ since 
$G_Y\cap D_Y=\emptyset$, 
and hence
$G\cap\wi{E}\neq\emptyset$ by \ref{int}.

Let us consider now the image $\wi{\sigma}(G)\subset\wi{Y}$. Then $\wi{\sigma}(G)$ is a section of $\wi{\pi}$, so that $\wi{\sigma}(G)\cong Z$
and $\rho_{\wi{\sigma}(G)}=1$. Moreover $\wi{\sigma}(G)$
contains the 
center $A_{\wi{Y}}\subset \wi{Y}$ of the blow-up $\wi{\sigma}$.
This also implies  that $\wi{Y}$ is Fano, 
as in the proof of Lemma \ref{uno}. 
\end{pargtwo}
\begin{pargtwo}
Suppose now that there exists a section $H\subset\wi{Y}$ of
$\wi{\pi}\colon \wi{Y}\to Z$, disjoint from 
$\wi{\sigma}(G)$. Then its strict  transform in $Y$ yields a section of   $\pi\colon Y\to Z$, disjoint from $G_Y$, and containing $A_Y$, and this implies the statement by \ref{secondsectionY}.

In order to construct such $H$,
we consider the divisor  $D_{\wi{Y}}:=\wi{\sigma}(D)\subset \wi{Y}$. Notice that 
$\dim\N(D_{\wi{Y}},\wi{Y})=1$, and that the two divisors $D_{\wi{Y}}$ and $\wi{\sigma}(G)$ are distinct, because $G\cap D=\emptyset$ in $X$.
\end{pargtwo}
\begin{pargtwo}\label{secondsection}
Suppose first that  $D_{\wi{Y}}$ is not nef. 
By Remark \ref{nonnefdiv}, $D_{\wi{Y}}$ is the exceptional locus of a divisorial contraction sending $D_{\wi{Y}}$ to a point.
 Then, by Lemma \ref{section_bundle}, $D_{\wi{Y}}$ is a section of $\wi{\pi}$.

Moreover we have $D_{\wi{Y}}\cdot C\geq 0$ for every curve $C\subset \wi{\sigma}(G)$, because $D_{\wi{Y}}\neq\wi{\sigma}(G)$ and $\rho_{\wi{\sigma}(G)}=1$. Since 
 $D_{\wi{Y}}$ is not nef, the divisors $D_{\wi{Y}}$ and $\wi{\sigma}(G)$ must be disjoint, and we can set $H:=D_{\wi{Y}}$.
\end{pargtwo}
\begin{pargtwo}
We assume  now that $D_{\wi{Y}}$ is nef. 
Then 
Proposition \ref{divisorial} applies,  and $D_{\wi{Y}}^{\perp}\cap\NE(\wi{Y})$ is an extremal ray $S_{\wi{Y}}$ of $\NE(\wi{Y})$. 

We claim that case $(iii)$ of Proposition \ref{divisorial} cannot happen, namely that $S_{\wi{Y}}$ cannot be small. Indeed this follows from Proposition \ref{divisorial} if $A_{\wi{Y}}\subset D_{\wi{Y}}$, namely if  $D\cap \wi{E}\neq\emptyset$. If instead $D\cap \wi{E}=\emptyset$, then
$\wi{\sigma}(G)\cap D_{\wi{Y}}=\emptyset$, hence 
 the contraction of $S_{\wi{Y}}$ sends $\wi{\sigma}(G)$ to a point, and it cannot be small.

Suppose that we are in case $(i)$ of Proposition \ref{divisorial}. As in \ref{applydivisorial} we see that $\wi{Y}\cong Z\times\pr^1$, and  
$D_{\wi{Y}}$ and  $\wi{\sigma}(G)$ are fibers of the projection
 $\wi{Y}\to \pr^1$. Thus we can define $H$ to be
 a general fiber of the projection $\wi{Y}\to\pr^1$.

Finally, suppose  that  we are in case $(ii)$ of Proposition \ref{divisorial}, so that
the contraction of $S_{\wi{Y}}$ is divisorial. By Lemma \ref{section_bundle}, $\Lo(S_{\wi{Y}})$ is a section of $\wi{\pi}$. So if $\Lo(S_{\wi{Y}})\cap\wi{\sigma}(G)=\emptyset$, we set $H:=\Lo(S_{\wi{Y}})$.

If instead $\Lo(S_{\wi{Y}})\cap\wi{\sigma}(G)\neq\emptyset$, then we get 
 $\Lo(S_{\wi{Y}})=\wi{\sigma}(G)$, because 
$\rho_{\wi{\sigma}(G)}=1$.  Therefore by Lemma \ref{section_bundle}, there is a section $H$ of $\wi{\pi}$ disjoint from $\wi{\sigma}(G)$. 
\qed
\end{pargtwo}

\smallskip

The following result is certainly well-known to experts. We include a proof for lack of references.
\begin{lemma}\label{section_bundle}
Let $Y$ and $Z$ be smooth connected projective varieties, and let 
$\pi : Y \to Z$ be a $\pr^1$-bundle. Let $\psi : Y \to Y_0$ be a birational morphism onto a projective variety 
sending an effective and reduced divisor $G$ to points. Then 
$Y\cong\mathbb{P}_Z(\ol_Z\oplus\mathscr{M})$ for some 
line bundle 
$\mathscr{M}$ on $Z$ so that $G$ identifies with the section of 
$\pi$ corresponding to $\ol_Z\oplus\mathscr{M}\twoheadrightarrow\mathscr{M}$. Moreover, 
$\mathscr{M}^{\otimes -1}$ is ample.
\end{lemma}
\begin{proof}
By replacing $\psi$ with its Stein factorization, we may assume that $Y_0$ is normal and that $\psi$ has connected fibers.

We show that $G$ is a section of $Y \to Z$.
Notice that $\pi$ is finite on $G$. 

Let $B \subset Z$ be a general smooth connected curve, and set $S:=\pi^{-1}(B)$. Then $G\cap S$ is a reduced curve;
let $C$ be an irreducible component of $G\cap S$. Moreover let $C_0\subset S$ be a minimal section, and $f\subset S$ a fiber of $\pi$. Then $C\neq f$. Set $e=-C_0^2$.
 
Suppose that $C\neq C_0$. 
Then
$C\equiv aC_0+bf$ where $a \in \mathbb{Z}_{>0}$, 
$b\ge ae$ if $e\ge 0$, and $2b\ge ae$ if $e< 0$ (see \cite[Propositions V.2.20 and V.2.21]{hartshorne}). Thus:
$$C^2=(aC_0+bf)^2=-a^2e+2ab=a(-ae+2b)\ge 0.$$
On the other hand, the restriction of $\psi$ to $S$ induces a birational morphism $\psi_{|S}\colon S \to \psi(S)$ sending $C$ to a point, and hence $C^2<0$, yielding a contradiction. 
Thus $C=C_0$, hence $G\cap S=C_0$.
This completes the proof of the first assertion. 

Set $\mathscr{G}:=\pi_*\ol_Y(G)$. Then $\mathscr{G}$ is a locally free sheaf of rank $2$ that fits into a short exact sequence
$$0 \to \ol_Z \to \mathscr{G} \to \mathscr{M} \to 0$$
corresponding to a class $\alpha \in H^1(Z,\mathscr{M}^{\otimes -1})$,  
$Y$ identifies with $\mathbb{P}_Z(\mathscr{G})$, $\ol_Y(G)$ with the tautological line bundle $\ol_{\mathbb{P}_Z(\mathscr{G})}(1)$,
and $G$ corresponds to $\mathscr{G} \twoheadrightarrow \mathscr{M}$. 

We denote by $2G$ the non-reduced closed subscheme of $Y$ defined by the ideal sheaf $\ol_Y(-2G)$.
By \cite[Lemme 3.2 and Lemme 3.3]{druel04} we have 
$$\textup{Pic}(G)\oplus H^1(Z,\mathscr{M}^{\otimes -1})\cong  \textup{Pic}(2G)$$
and,  under the above isomorphism,  $(0,\alpha)$ maps to
the class of  $(\ol_{\mathbb{P}_Z(\mathscr{G})}(1)\otimes\pi^*\mathscr{M}^{\otimes -1})_{|2G}$.

Let $\mathscr{H}$ be an ample line bundle on $Y_0$. Then there exists 
$m\in\mathbb{Z}_{>0}$ such that 
$\psi^*\mathscr{H}\cong\ol_{\mathbb{P}_Z(\mathscr{G})}(m)\otimes\pi^*\mathscr{M}^{\otimes -m}$. This implies that
$(\ol_{\mathbb{P}_Z(\mathscr{G})}(m)\otimes\pi^*\mathscr{M}^{\otimes -m})_{|2G}\cong\ol_{2G}$. Hence we must have $\alpha = 0$, and 
$\mathscr{G}\cong \ol_Z\oplus\mathscr{M}$. Let $G'$ be the section of $Y \to Z$ corresponding to $\mathscr{G}\cong \ol_Z\oplus\mathscr{M}\twoheadrightarrow \ol_Z$. Then $G\cap G'=\emptyset$, therefore ${\psi^*\mathscr{H}}_{|G'}$ is ample.
But ${\psi^*\mathscr{H}}_{|G'}\cong \mathscr{M}^{\otimes -m}$ under the isomorphism $G'\cong Z$. This completes the proof of the lemma.
\end{proof}
\section{Fano manifolds with a locally unsplit family of rational curves of anticanonical degree $n$}\label{proofs}
In this section, we prove Theorem \ref{result} and Proposition \ref{prop_main}.
We start with the following observations.
\begin{remark}\label{locuns}
Let $V$ be a locally unsplit dominating family of rational curves on a smooth projective variety $X$ of dimension $n$, and suppose that the curves of the family have anticanonical degree $n$.  Let $x$ be a general point, and let $\Lo(V_x)\subseteq X$ be the union of all curves parametrized by $V_x$.
Then by \cite[Corollaries IV.2.6.3 and II.4.21]{kollar} we have that  $\Lo(V_x)$ is a divisor and $\N(\Lo(V_x),X)=\R[V]$.
\end{remark}
\begin{remark}\label{deflX}
Let $X$ be a Fano manifold, and recall that  $l_X$
is the minimal anticanonical degree of a locally unsplit dominating family of rational curves in $X$.

If $x\in X$ is a general point, then every irreducible rational curve $C$ through $x$ has anticanonical degree at least $l_X$, see \cite[Theorem IV.2.4]{kollar}. This implies that $l_X$ can equivalently be defined as the minimal anticanonical degree of a dominating family of rational curves in $X$.
\end{remark}
\begin{proof}[Proof of Theorem \ref{result}]
Since $l_X=n$, there is a locally unsplit dominating family $V$ of rational curves of 
 anticanonical degree $n$. Thus by Remark \ref{locuns} $X$ contains a prime divisor $D$ with $\dim\N(D,X)=1$,
 so that we can apply Lemma \ref{uno}.

Since $l_X=n>2$, 
 we know that $X$ cannot have a conic bundle structure. Therefore Lemma \ref{uno} yields that $\rho_X=2$ and 
 there exists $\sigma\colon X\to Y$ such that $Y$ is smooth with $\dim Y=n$ and $\rho_Y=1$, and $\sigma$ is the blow-up of $A\subset Y$ smooth of codimension $2$.

Let $E\subset X$ be the exceptional divisor; we have $-K_X+E=\sigma^*(-K_Y)$.

Let $W$ be a locally unsplit dominating family of rational curves in $Y$, $C\subset Y$ a general curve of the family, and $\widetilde{C}\subset X$ its 
strict
transform. By \cite[Proposition II.3.7]{kollar}, 
$C\cap A = \emptyset$, and hence $E\cdot\widetilde{C}= 0$. Moreover 
$\widetilde{C}$ moves in a dominating family $\w{W}$ of rational curves in $X$, so that $-K_X\cdot \widetilde{C}\geq n$.
This yields $-K_Y\cdot C\geq n$.

We show that  $-K_Y\cdot [W]> n$.  We argue by contradiction, and assume that $-K_Y\cdot[W]=n$. Consider a curve $C$ as above. Then we have $-K_X\cdot \w{C}=n$. Since $-K_X\cdot[\w{W}]=n=l_X$, the family $\w{W}$ is locally unsplit. 

Now for a general point $x_1\in X$, again by Remark \ref{locuns}, $\Lo(\w{W}_{x_1})$ is a divisor with  $\N(\Lo(\w{W}_{x_1}),X)=\R[\w{W}]$. Since $E\cdot[\w{W}]=0$, the divisor 
$\Lo(\w{W}_{x_1})$ must be disjoint from $E$, thus its image in $Y$ is disjoint from $A$. 
This gives a contradiction, because every non-trivial effective divisor in $Y$ is ample 
 (recall that $\rho_Y=1$).

We conclude that $W$ has anticanonical degree $n+1$, 
so that $Y\cong\pr^n$ by Theorem \ref{CMSBK}.

Now if $\ell\subset\pr^n$ is a line intersecting $A$ in at least two points, and $\w{\ell}\subset X$ is its strict
transform, then $-K_X\cdot\w{\ell}<n$. We deduce that the secant variety $\Sec(A)$ of $A$ is not the whole $\pr^n$, and this implies that $A$ is degenerate (see for instance
\cite[Theorem 3.4.26]{lazI}). Finally one can check that since $X$ is Fano, the degree $d$ of $A$ is at most $n$, completing the proof.
\end{proof}
We need some preliminary results for Proposition \ref{prop_main}, which is
the main result of this section.
\begin{lemma}\label{inducedfamily}
Let $X$ and $Y$ be smooth projective varieties, and let 
$\pi \colon X \to Y$ be a surjective morphism with connected fibers. 
Let $V$ be a locally unsplit dominating family of rational curves on 
$X$. Suppose that $\pi$ does not contract any curve from $V$.
\begin{enumerate}[(1)]
\item If  $K_{X/Y}\cdot[V]\leq 0$, then  $K_{X/Y}\cdot[V]=0$.
\item Suppose moreover that $-K_{X}\cdot [V]= \dim Y+1$, and that for a general $x\in X$, and for every $[C]\in V_x$, the map $\pi_{|C}\colon C\to \pi(C)$ is birational.
Then  $Y\cong \pr^{\dim Y}$.
\end{enumerate}
\end{lemma}
\begin{proof}
Let $x$ be a general point in $X$, and let $[C]\in V_x$ be a general curve. Set $\ell:=\pi(C)$, and let 
$m$ be the positive integer such that $\pi_* C=m \ell$.
Notice that $\ell$ is a free curve and hence yields a smooth point in $\Rat(Y)$. Let 
$V_Y$ be the irreducible component of $\Rat(Y)$ which contains $[\ell]$.
We set $y:=\pi(x)$. We also set $V_x':=\{[C]\in V\,|\,x\in C\}\subseteq V$, and similarly 
we define $(V_Y)_y'\subseteq V_Y$ (recall that $V_x$ is the normalization of $V_x'$, and 
 $(V_Y)_y$ that of $(V_Y)_y'$). Notice that $V_x'$ has pure dimension $-K_X\cdot[V]-2$, and since $V_Y$ is a dominating family, 
 $(V_Y)_y'$ has pure dimension $-K_Y\cdot [V_Y]-2$.

Let us consider the morphism $\pi_*\colon V \to \textup{Chow}(Y)$
induced by
the push-forward morphism $\textup{Chow}(X)\to \textup{Chow}(Y)$ (see \cite[Theorem I.6.8]{kollar}).
We claim that
$\pi_*$ is finite on $V_x'$. 
Suppose otherwise, and let 
$T \subseteq V_x'$ be an irreducible complete curve contained in a fiber of $\pi_*$. We set $\Lo(T):=\cup_{t\in T} C_t\subseteq X$. Then $\dim\Lo(T)= 2$, and $\dim\N(\Lo(T))=1$ by \cite[Corollary II.4.21]{kollar}. This implies that $\pi$ is finite on $\Lo(T)$, thus $\dim\pi(\Lo(T))=2$. On the other hand $\ell':=\pi(C_t)$ does not depend on $t \in T$, because $T$ is contained in a fiber of $\pi_*$. Hence $\pi(\Lo(T))=\ell'$, a contradiction.
This proves that $\pi_*$ is finite on $V_x'$.

Therefore, the rational map $V_x' \dashrightarrow (V_Y)_{y}'$ sending
$[C]\in V_x'$ to $\frac{1}{m}\pi_*[C]=[\ell] \in (V_Y)_{y}'$ 
is generically finite, and hence $\dim (V_Y)_{y}'\ge \dim V_x'$. We get:
$$-K_Y\cdot [V_Y] =\dim (V_Y)_{y}'+2\ge \dim V_x' +2=-K_X \cdot [V].$$

Suppose from now on that $K_{X/Y}\cdot [V] \le 0$. Then 
$$-K_X\cdot [V] \ge -mK_Y \cdot [V_Y],$$ and thus
$m=1$ and $-K_X\cdot [V]=-K_Y \cdot [V_Y]$, proving $(1)$.

We proceed to prove $(2)$. Since $-K_Y \cdot [V_Y]=\dim Y +1$,
 by Theorem \ref{CMSBK}
it is enough to show that $V_Y$ is a
locally unsplit family of rational curves on $Y$. 
Let $X_0$ be a dense open subset of $X$ such that for every point $x \in X_0$
\begin{itemize}
\item for any curve $[C]\in V_x$, $C$ is a free curve, 
\item $V_x$ is proper, and 
\item for any curve $[C]\in V_x$, $\pi_{|C}\colon C\to \pi(C)$ is birational.
\end{itemize}
Let $V_0$ be the open subset of $V$ consisting of points $[C]$ such that $C\cap X_0\neq \emptyset$.
Observe that $(\pi_*)_{|V_0}\colon V_0 \to \textup{Chow}(Y)$ factors through $V_Y \to \textup{Chow}(Y)$. We still denote by $(\pi_*)_{|V_0}$ the induced morphism
$V_0 \to V_Y$. 

Since $V_x'$ is proper, to prove that $V_Y$ is a locally unsplit family of rational curves, it
is enough to show that
$(\pi_*)_{|V_x'} \colon V_x' \to (V_Y)_{y}'$ is dominant.

Consider the universal families $U_0\to V_0$ and $U_Y\to V_Y$, and the evaluation morphisms 
$e\colon U_0\to X$
and
$e_Y\colon U_Y\to Y$. Notice that $e$ is flat by \cite[Corollary II.3.5.3 and Theorem II.2.15]{kollar}.
We have a commutative diagram:
$$
\xymatrix{ X\ar[d]_{\pi} & {U_0}\ar[d]^{\mu} \ar[l]_{e}\ar[r]^p&
{V_0}\ar[d]^{{(\pi_*)}_{|V_0}}\\
Y&{U_Y}\ar[r]^{p_Y}\ar[l]_{e_Y}&{V_Y}}
$$
Recall from the proof of (1) that
$\dim (\pi_*)_{|V_0}(V_x')=\dim (V_Y)_y'$. This implies that
$(\pi_*)_{|V_0}$ is dominant, and hence so is $\mu$. 

On the other hand,
since $V_x'=p(e^{-1}(x))$ and $(V_Y)_y'=p_Y(e_Y^{-1}(y))$, we also have 
a commutative diagram:
$$
\xymatrix{ 
e^{-1}(x) \ar[r]\ar[d]_{\mu_{|e^{-1}(x)}} & {V_x'} \ar[d]^{{(\pi_*)}_{|V_x'}}\\
e_Y^{-1}(y)\ar[r]& {(V_Y)_y'}
}
$$
where the horizontal arrows are finite and surjective, and $(\pi_*)_{|V_x'}$ is finite by the proof of (1). Therefore $\mu_{|e^{-1}(x)}\colon e^{-1}(x)\to e_Y^{-1}(y)$ is finite, and 
to show that $(\pi_*)_{|V_x'}\colon V_x' \to (V_Y)_{y}'$ is dominant, it is enough to show that $\mu_{|e^{-1}(x)}\colon e^{-1}(x)\to e_Y^{-1}(y)$ is dominant. 

Consider now $Z\subset U_0$ a general fiber of the composition $\pi\circ e\colon U_0\to Y$; $Z$ has pure dimension $\dim U_0-\dim Y$.
$$\xymatrix{{\pi^{-1}(y)} &  Z\ar[l]_{e_{|Z}}\ar[d]^{\mu_{|Z}}\\& e_Y^{-1}(y)}$$
Let $F$ be an irreducible component of $e_Y^{-1}(y)$, and let $Z_0$ be an irreducible component of $Z$ dominating $F$ under $\mu$. Since $e$ is flat, $Z_0$ must also dominate $\pi^{-1}(y)$ under $e$. Thus for $x\in\pi^{-1}(y)$ general, $e^{-1}(x)\cap Z_0$ is non-empty and has pure dimension 
$\dim Z_0-\dim\pi^{-1}(y)=\dim e^{-1}(x)=\dim e_Y^{-1}(y)=\dim F$. Finally $\mu$ is finite on  $e^{-1}(x)\cap Z_0$, so that $e^{-1}(x)\cap Z_0$ dominates $F$ under $\mu$.

We conclude that $\mu_{|e^{-1}(x)}\colon e^{-1}(x)\to e_Y^{-1}(y)$ is dominant, completing the proof of the lemma.
\end{proof}
\begin{lemma}\label{scb}
Let $X$ and $Y$ be smooth projective varieties, and let 
$\pi \colon X \to Y$ be a $\pr^1$-bundle. Let $V$ be a locally unsplit dominating family of rational curves on $X$ with $-K_X\cdot [V]=\dim(X)\ge 3$. Set $n:=\dim(X)$.
Then $X\cong \pr^1\times \pr^{n-1}$, and $V$ is the family of lines in the $\pr^{n-1}$'s.
\end{lemma}
\begin{proof}
Let $x$ be a general point in $X$, $V'_x$ an irreducible component of $V_x$, and $U'_x \to V'_x$ the universal
family. Let 
$U'_x \to X$ be the evaluation morphism, and
$U'_x \to T$ its Stein factorization.
Then $T$ is a normal generalized cone; in particular $\rho_T=1$ by 
\cite[Corollary II.4.21]{kollar}.

We claim that the composite map $T \to X \to Y$ is finite; in particular, it is dominant. Suppose otherwise.
Then, by \cite[Corollary II.4.21]{kollar}, $\pi$ sends every curve in
$D:=\Lo(V'_x)$ to a point.
Thus $-K_X\cdot [V]=2$, yielding a contradiction.

Set $Z:=T\times_Y X$, with natural morphisms
$\tau\colon Z\to T$, and $\nu\colon Z\to X$. 
$$\xymatrix{Z\ar[r]^{\nu}\ar[d]_{\tau} & X\ar[d]^{\pi}\\
T\ar[r]&Y
}$$
Notice that $\tau$ is a $\pr^1$-bundle. 
Let $T_Z\subset Z$ be the section of $\tau$ induced by $T\to X$; 
we have
$\nu(T_Z)=D$. 
Then $Z\cong \mathbb{P}_{T}(\mathscr{E})$,
where $\mathscr{E}:=\tau_*\mathscr{O}_Z(T_Z)$ is a rank 2 vector bundle on $T$ that fits into a short exact sequence
$$0\to \mathscr{O}_T\to \mathscr{E}\to \mathscr{M}\to 0,$$
with $\mathscr{M}$ a line bundle on $T$. Moreover,
$\mathscr{E}\twoheadrightarrow \mathscr{M}$ corresponds to the section $T_Z$, and $\mathscr{O}_Z(T_Z)$ identifies with the tautological line bundle, so that $\mathscr{O}_Z(T_Z)_{|T_Z}\cong\tau^*\mathscr{M}_{|T_Z}$.

We prove that $K_{X/Y}\cdot [V]\le 0$.
Suppose otherwise.
We have:
$$\mathscr{O}_Z(K_{Z/T})\cong\mathscr{O}_Z(-2T_Z)\otimes
\tau^*\mathscr{M},$$
and therefore
$$\mathscr{O}_Z(K_{Z/T})_{|T_Z}\cong\mathscr{O}_Z(-T_Z)_{|T_Z}.$$

Let $[C]\in V'_x$, and let 
$C_Z$ be an irreducible component of $\nu^{-1}(C)$ contained in $T_Z$.
Then, by the projection formula:
$$-T_Z\cdot C_Z=K_{Z/T}\cdot C_Z = 
\nu^*(K_{X/Y}) \cdot C_Z = mK_{X/Y}\cdot C>0,$$
where $m\in\mathbb{Z}_{>0}$ is such that $\nu_*C_Z=mC$.
This implies that $\mathscr{M}\cdot \tau_*C_Z<0$.

Let now $[C']\in V$ be a general point, and $C'_Z$ an irreducible component of $\nu^{-1}(C')$ not  contained in $T_Z$. 
Then, as above, we must have $K_{Z/T}\cdot C'_Z>0$. On the other hand, $T_Z\cdot C'_Z\ge 0$ since $C'_Z\not\subset T_Z$, and
$\mathscr{M}\cdot \tau_*C'_Z<0$ because $\rho_T=1$.
This implies that
$$K_{Z/T}\cdot C'_Z=-2T_Z\cdot C'_Z+\mathscr{M}\cdot \tau_*C'_Z<0,$$
yielding a contradiction.
Therefore $K_{X/Y}\cdot [V]\le 0$, and Lemma \ref{inducedfamily}(1) 
yields $K_{X/Y}\cdot [V]=0$.

Let $x \in X$ be a general point, $[C]\in V_x$,
and set $y:=\pi(x)$ and $\ell := \pi(C)$. We show that
$\pi_{|C}:C\to\ell$ is birational.
Let $\mathbb{F}$ be the normalization of $\pi^{-1}(\ell)$, and let $\tilde\ell$ be the normalization of $\ell$. Let $C'$ be the strict transform of $C$ in $\mathbb{F}$. By the projection formula, we have
$K_{\mathbb{F}/\tilde\ell}\cdot C'=K_{X/Y}\cdot C=0$. This implies that
$\mathbb{F}\cong \tilde \ell\times\pr^1$, and that $C'$ is a fiber of $\mathbb{F} \to \pr^1$, proving our claim.

Therefore
$Y\cong \mathbb{P}^{n-1}$ by Lemma \ref{inducedfamily}(2), and $T=D$ is a section of $\pi$.
Since $K_{X/Y}\cdot [V]=0$, we must have $\mathscr{M}\equiv 0$, and hence 
$\mathscr{M}\cong\mathscr{O}_{\mathbb{P}^{n-1}}$. Finally, $\mathscr{E}\cong \mathscr{O}_{\mathbb{P}^{n-1}}^{\oplus 2}$, since 
$h^1(\mathbb{P}^{n-1},\mathscr{O}_{\mathbb{P}^{n-1}})=0$. This completes the proof of the lemma.
\end{proof}
\begin{lemma}\label{contraction}
Let $X$ be a Fano manifold of dimension $n\geq 3$ and Picard number $\rho_X=2$, and suppose that $X$ has a locally unsplit dominating family of rational curves $V$ of anticanonical degree $n$. Let $D\subset X$ be an irreducible component of $\Lo(V_x)$ for a general point $x\in X$.

Then one of the following holds:
\begin{enumerate}[$(a)$]
\item $X$ is the blow-up of $\pr^n$ along a linear subspace of codimension $2$;
\item $X\cong\pr^1\times\pr^{n-1}$;
\item 
$D^{\perp}\cap\NE(X)$ is an extremal ray of $X$, whose associated contraction is divisorial and sends its exceptional divisor to a point; the other extremal contraction of $X$ is either a (singular) conic bundle, or the blow-up of a smooth variety along a smooth subvariety of codimension $2$. 
\end{enumerate}
\end{lemma}
\begin{remark}
The only examples known to the authors of Fano manifolds with dimension  $n\geq 3$, Picard number $\rho_X=2$, and having a locally unsplit dominating family of rational curves of anticanonical degree $n$, are $\pr^1\times\pr^{n-1}$ and the varieties from  Example \ref{exampleM}. These are obtained as the blow-up of $\pr^n$ along a smooth subvariety $A$, of dimension $n-2$ and degree   $d\in\{1,\dotsc,n\}$, contained in a hyperplane. If $d=1$, this is case $(a)$ of Lemma \ref{contraction}. If $d>1$, this gives an example of case $(c)$ of Lemma \ref{contraction}.  
\end{remark}
\begin{proof}[Proof of Lemma \ref{contraction}]
By Remark \ref{locuns}, $\N(D,X)=\R[V]$. This implies that $D$ is nef since $D\cdot [V]\geq 0$. Let $R$ be an extremal ray of $X$ such that $D\cdot R>0$, and let $\sigma\colon X\to Y$ be the contraction of $R$.
By Lemma \ref{uno}, $Y$ is a smooth Fano variety, and either $\sigma$
is a blow-up with center $A_Y\subset Y$ smooth of codimension 2, or $\sigma$ is a conic bundle.
Set $S:=D^{\perp}\cap\NE(X)$. Then $S$ is an extremal ray of $X$ by Proposition \ref{divisorial}. 

Suppose that we are in case $(i)$ of Proposition \ref{divisorial}. 
We show that we are in case $(a)$ or $(b)$.
The contraction of $S$ is $\ph\colon X\to\pr^1$, $D$ is a fiber, and hence 
$D\cdot [V]=0$. Let 
$F$ be a general fiber of $\ph$.
Then the family of rational curves from $V$ contained in $F$ is locally unsplit with anticanonical degree $\dim(F)+1$, thus $F\simeq\mathbb{P}^{n-1}$ by Theorem \ref{CMSBK}. 

Let $B_0\subset \pr^1$ be a dense open subset such that
$X_0:=\ph^{-1}(B_0) \to B_0$ is smooth. 
By Tsen's Theorem, 
there exists a divisor $H_0$ on $X_0$ such that $\mathscr{O}_F({H_0}_{|F})\cong \mathscr{O}_{\mathbb{P}^{n-1}}(1)$. 

Let $H$ be the closure of $H_0$ in $X$. Then $H$ is $\ph$-ample since $\ph$ is elementary.
If $p\in\pr^1$, then 
$[\ph^*(p)]\cdot H^{n-1}=1$. Therefore, 
all fibers of $\ph$ are integral, and
by \cite[Corollary 5.4]{fujita75} we have
 $X\cong \mathbb{P}_{\mathbb{P}^1}(\mathscr{E})$ with
$\mathscr{E}$ a vector bundle of rank $n$ on $\pr^1$.
Since $X$ is Fano, it is not difficult to see that
either $X\cong \pr^1\times\pr^{n-1}$, or  $X$ is the blow-up of $\pr^n$ along a linear subspace of codimension $2$. Thus we get $(a)$ or $(b)$.

Case $(ii)$ of  Proposition \ref{divisorial} is $(c)$. 

We show that case $(iii)$ of Proposition \ref{divisorial} does not occur.
Suppose otherwise. Then the contraction $\ph$ of $S$ is small, it has a flip $X\dasharrow X'$, $X'$ is smooth, and there is a smooth $\pr^1$-bundle $\psi\colon X'\to Y'$. 
Let $[C]\in V$ be a general point. Then 
$\textup{Exc}(\ph)\cap C=\emptyset$ (see \cite[Proposition II.3.7]{kollar}). Let $V'$ be the 
irreducible component of $\Rat(X')$ which contains $C'$
the strict transform of $C$ in $X'$. 

We show that 
$V'$ is a locally unsplit dominating family of rational curves on $X'$. Let $x\in X\setminus \textup{Exc}(\ph)$ be a general point. If $\Lo(V_x)\cap \textup{Exc}(\ph)\neq \emptyset$, then 
$[V]\in S$ since $\N(\Lo(V_x),X)=\R[V]$ and $D\cdot S=0$, and hence $\Lo(S)=X$, yielding a contradiction.
Therefore, $\Lo(V_x)\cap \textup{Exc}(\ph)= \emptyset$, and hence
$V'_x\cong V_x$ is proper. This proves that $V'$ is a locally unsplit family of rational curves. Notice also that
$-K_{X'}\cdot [V']=n$. By Lemma \ref{scb}, $X'\simeq \pr^1\times \pr^{n-1}$, so that $X'$ does not have small contractions, a contradiction.
This completes the proof of the lemma.
\end{proof}
Finally, we prove the main result of this section.
\begin{proposition}\label{prop_main}
Let $X$ be a Fano manifold of dimension $n\geq 3$, and suppose that $X$ has a locally unsplit dominating family $V$ of rational curves of anticanonical degree $n$. Then $\rho_X\leq 3$. 

If moreover $\rho_X=3$, then $X$ is isomorphic to one of the varieties described in Example \ref{ex}, and $[V]\equiv 
C_{\wi{G}}+(d-a)F$, where $F$ is a fiber of $\sigma$, and $C_{\wi{G}}$ is the  strict transform of a line in $\wi{G}_Y\cong\pr^{n-1}$ (notations as in Example \ref{ex}).
\end{proposition}

\noindent\emph{Proof.} Let $D\subset X$ be an irreducible component of $\Lo(V_x)$ for a general point $x\in X$. Then $\N(D,X)=\R[V]$ by Remark \ref{locuns}, and by Lemma \ref{uno}  we have $\rho_X\leq 3$. 
Notice that $D$ is nef, because  $D\cdot [V]\geq 0$. 
\begin{pargtwo}
We assume that $\rho_X=3$. By Theorem \ref{main2}, 
$X$ is as described in Example \ref{ex2}. We use the notations of Example \ref{ex2} and Theorem \ref{main2}; in particular we refer the reader to Table \ref{table} for the intersection table of $X$.
We have to show that $Z\cong\pr^{n-1}$, that $a\leq d$, and that $[V]\equiv C_G+a\wi{F}\equiv C_{\wi{G}}+(d-a)F$ (see \eqref{equiv}).
\end{pargtwo}
\begin{pargtwo}\label{secondpart}
Notice that $G\cdot [V]=0$. Indeed if for instance $G\cap D\neq \emptyset$, as $\rho_G=1$, then there exists $\lambda\in\Q_{>0}$ such that $[V]=\lambda [C_G]$. On the other hand  $G\cdot [V]\geq 0$ while  $G\cdot C_G=-a\delta\leq 0$, and we conclude that $G\cdot [V]=0$.

Set $m:=-K_Z\cdot C_Z=i_Z\delta>0$. 
Since $[F],[\wi{F}],[C_G]$ are a basis of $\N(X)$, and we can write:
$$[V]\equiv\alpha F+\beta\wi{F}+\gamma C_G,$$
with $\alpha,\beta,\gamma\in\Q$. 
Intersecting with $G$ yields $\beta=a\gamma\delta$, and intersecting with $-K_X$ yields $\alpha=n-m\gamma$:
\renewcommand{\theequation}{\thepargtwo}
\stepcounter{pargtwo}\begin{equation}\label{[V]}
[V]\equiv(n-m\gamma) F+a\gamma\delta\wi{F}+\gamma C_G.\end{equation}
Notice that $\gamma>0$, because $[F]$ belongs to the extremal ray $R$, and $[V]\not\in R$. 

Moreover $\wi{E}\cdot [V]=\gamma\delta(d-a)$, so we obtain $d-a\geq 0$.

Finally, we also have $\wi{G}\cdot [V]=n-m\gamma\geq 0$.
\end{pargtwo}
\begin{pargtwo}\label{base}
By \eqref{[V]} we have
$$
K_{X/Z}\cdot[V]=K_X\cdot[V]-K_Z\cdot\ph_*[V]=-n-K_Z\cdot \gamma C_Z=-n+\gamma m\leq 0.
$$
Then Lemma \ref{inducedfamily}(1) 
yields $K_{X/Z}\cdot[V]=0$ and hence $\gamma=n/m$, and
 \eqref{[V]} becomes:
\renewcommand{\theequation}{\thepargtwo}
\stepcounter{pargtwo}\begin{equation}\label{V}
[V]\equiv \frac{n}{m}\left(a\delta\wi{F}+C_G\right).\end{equation}
\end{pargtwo}
\begin{pargtwo}
Let $x\in X$ be a general point.
We show that $\ph$ has degree $1$ on every curve of $V_x$.

We proceed by
contradiction, and assume that
for some $[C_\infty] \in V_x$, the 
morphism $C_\infty \to \ph(C_\infty)$ has degree $k\ge 2$.

Let $\ell\to\ph(C_\infty)$ be the normalization, 
set $\mathbb{F}:=\ell\times_Z Y$, and denote by 
${\pi}_{|\mathbb{F}}\colon \mathbb{F}\to\ell$ 
(respectively, $\nu\colon \mathbb{F}\to Y$) 
the natural morphisms.
$$\xymatrix{{\mathbb{F}}\ar[r]^{\nu}\ar[d]_{\pi_{|\mathbb{F}}} & 
{Y}\ar[d]_{\pi} &X\ar[l]_{\sigma}\ar[dl]^{\ph}\\
{\ell}\ar[r]&{Z}&
}$$
Let $C_\infty'\subset \mathbb{F}$ be the  strict transform of $\sigma(C_\infty)\subset Y$.
Let moreover $C_0\subset \mathbb{F}$ be a minimal section of 
${\pi}_{|\mathbb{F}}$, and $f\subset \mathbb{F}$ a general fiber.

We have $G_Y\cdot\sigma(C_\infty)=\sigma^*(G_Y)\cdot C_\infty=G\cdot[V]=0$ by \ref{secondpart}, and $\sigma(C_\infty)\not\subseteq G_Y$ because $x\in C_\infty$ is general, hence  $G_Y\cap\sigma(C_\infty)=\emptyset$.

The pull-back of $G_Y$ to $\mathbb{F}$ is precisely $C_0$, so that $C_\infty'\cap C_0=\emptyset$. Moreover ${\pi}_{|\mathbb{F}}$ has degree $k$ on $C_\infty'$.

Set $a_0:=-C_0^2\in\mathbb{Z}_{\geq 0}$. Then 
 $\mathbb{F}\cong \mathbb{F}_{a_0}$, and we obtain: 
\renewcommand{\theequation}{\thepargtwo}
\stepcounter{pargtwo}\begin{equation}\label{linearsystem}
C_\infty'\equiv k\left(C_0+a_0f\right)\quad\text{and}\quad
-K_{\mathbb{F}}\cdot C_\infty'=k(a_0+2).\end{equation}

Consider now $\nu^*\wi{G}_Y\subset\mathbb{F}$. This is a section of $\pi_{|\mathbb{F}}$ disjoint from $C_0$, so that  $\nu^*\wi{G}_Y\equiv C_{0}+ a_0f$.
We also have $\sigma^*\wi{G}_Y=\wi{G}+E$, and $\wi{G}\cdot [V]=0$ by \eqref{V}, so using the projection formula:
\renewcommand{\theequation}{\thepargtwo}
\stepcounter{pargtwo}\begin{equation}\label{E}
E\cdot C_\infty=\sigma^*\wi{G}_Y\cdot C_\infty=\wi{G}_Y\cdot\sigma(C_\infty)=\nu^*\wi{G}_Y\cdot C_\infty'=ka_0.
\end{equation}

Since $A_Y\subset \wi{G}_Y$, we have $\nu^{-1}(A_Y)\subset \nu^*\wi{G}_Y$. Moreover $C_\infty\not\subseteq E\cup\wi{E}$ because $x\in C_\infty$, hence $\ph(C_\infty)\not\subseteq A$, and $\nu^{-1}(A_Y)$ is a zero-dimensional scheme.

We see $\nu^{-1}(A_Y)$ as a zero-dimensional subscheme of $\nu^*\wi{G}_Y\cong\pr^1$; in particular, it is determined by its support and by its multiplicity at each point. We write $\nu^{-1}(A_Y)=h_1 p_1+\cdots +h_r p_r$.
\end{pargtwo}
\begin{pargtwo}
Let $\varpi\colon\widetilde{\mathbb{F}}\to \mathbb{F}$ be the blow up of $\mathbb{F}$ along $\nu^{-1}(A_Y)$; 
note that there is a natural morphism 
$\w{\nu}\colon\widetilde{\mathbb{F}}\to X$:
$$\xymatrix{{\w{\mathbb{F}}}\ar[r]^{\w{\nu}}\ar[d]_{\varpi} &  X\ar[d]_{\sigma}\ar@/^1pc/[dd]^{\ph} \\
{\mathbb{F}}\ar[r]^{\nu}\ar[d]_{\pi_{|\mathbb{F}}} & 
{Y}\ar[d]_{\pi}\\
{\ell}\ar[r]&{Z}
}$$
Observe that 
$\widetilde{\mathbb{F}}$ is a normal surface, with local complete intersection singularities.
Indeed, locally over $p_i$,  $\widetilde{\mathbb{F}}$ can be described as the blow-up of $\mathbb{A}^2_{x,y}$ at the ideal $(x^{h_i},y)$. In particular, 
$\widetilde{\mathbb{F}}$ is smooth at $\varpi^{-1}(p_i)$ if $h_i=1$, otherwise it has a Du Val singularity  of type $A_{h_i-1}$. 

Set $F_i:=\varpi^{-1}(p_i)$ (with reduced scheme structure); then
\renewcommand{\theequation}{\thepargtwo}
\stepcounter{pargtwo}\begin{equation}\label{canonical}
K_{\widetilde{\mathbb{F}}}=\varpi^*K_{\mathbb{F}} +\sum_{1\le i\le r}h_i F_i,
\quad\text{and}\quad \sum_{1\le i\le r}h_i F_i=\w{\nu}^*E.
\end{equation}
Let $\w{C}'_\infty\subset\w{\mathbb{F}}$ be the  strict transform of $C'_\infty\subset\mathbb{F}$. Then $\w{\nu}(\w{C}'_\infty)=C_\infty$, and 
using \eqref{canonical}, \eqref{linearsystem}, and \eqref{E},
we get:
\renewcommand{\theequation}{\thepargtwo}
\stepcounter{pargtwo}\begin{equation}\label{antdegree}
-K_{\widetilde{\mathbb{F}}}\cdot \widetilde{C}'_\infty=
-K_{\mathbb{F}}\cdot C'_\infty
-E\cdot C_\infty
=2k+ka_0-ka_0=2k.
\end{equation}
\end{pargtwo}
\begin{pargtwo}
The curve $\w{C}'_\infty$ is an integral rational curve in $\w{\mathbb{F}}$, of anticanonical degree $2k$ by \eqref{antdegree}.
We fix a point $p_0\in \w{C}_\infty'$ such that $\w{\nu}(p_0)=x$.
By \cite[Theorems II.1.7 and II.2.16]{kollar},
every irreducible component of $\Rat(\w{\mathbb{F}},p_0)$ containing $[\w{C}'_\infty]$ has dimension 
$\ge 2k-2\ge 2$, because $k\ge 2$ by assumption.
 
Let $B\subseteq \Rat(\w{\mathbb{F}},p_0)$ be an irreducible curve containing $[\widetilde{C}'_\infty]$. 
If $B$ is proper, we get $\Lo(B)=\w{\mathbb{F}}$ and hence $\rho_{\w{\mathbb{F}}}=1$ by \cite[Corollary IV.4.21]{kollar}, a contradiction.
If instead $B$ is not proper, the closure of its image in $\Chow(\w{\mathbb{F}})$ contains points corresponding to non-integral curves. 
But this gives again a contradiction, because $V_x$ is proper.
 
We conclude that $Z\cong\pr^{n-1}$ by Lemma \ref{inducedfamily}(2), $C_Z\subset Z$ is a line, $\delta=\ol_Z(1)\cdot C_Z=1$, $m=-K_Z\cdot C_Z=n$, and $\gamma=1$. 
Finally $[V]\equiv a\wi{F}+C_G$ by \eqref{V}.
\qed
\end{pargtwo}
\section{Examples of  locally unsplit families of rational curves of anticanonical degree $n$}\label{families}
Let $X$ be one of the varieties introduced in Example \ref{ex}; we use the same notations as in Examples \ref{ex} and \ref{ex2}, with $Z=\pr^{n-1}$. 

We   construct a dominating family of rational curves on $X$, and then show that it is locally unsplit. Notice that the condition $a\leq d$  is necessary to ensure the existence of such a family, see \ref{secondpart}.

We first consider the case $a=0$, so that $Y=\pr^{n-1}\times\pr^1$, and $X$ is the blow-up of $\pr^{n-1}\times\pr^{1}$ along $A\times\{p_0\}$, where $p_0\in\pr^1$ is a fixed point.  The general curve of the family $V$ is the  strict transform in $X$ of $\ell\times\{p\}\subset{Y}$, where $p\neq p_0$ and $\ell\subset\pr^{n-1}$ is a line. 
Therefore $V$ is a locally unsplit dominating family of rational curves, and for a general point $x\in X$, $V_x$ is isomorphic to the variety of lines through a fixed point in $\pr^{n-1}$; hence $V_x\cong\pr^{n-2}$. 

Suppose from now on that $a>0$, and
set $$X_0:=X\smallsetminus\left(E\cup\wi{E}\cup G\cup \wi{G}\right).$$
Let $\ell \subset \pr^{n-1}$ be a line not contained in $A$, and let $x\in X_0$ be such that $\ph(x)\in\ell$.
Set
$\mathbb{F}:=\pi^{-1}(\ell)\cong \mathbb{P}_{\mathbb{P}^1}(\mathscr{O}_{\mathbb{P}^1}\oplus\mathscr{O}_{\mathbb{P}^1}(a))$, 
and denote by
$\pi_{|\mathbb{F}}$ the restriction of $\pi$ to $\mathbb{F}$. 
We have $G_Y\cap \mathbb{F}=C_0$ the minimal section of $\pi_{|\mathbb{F}}$, and $\wi{G}_Y\cap \mathbb{F}\cong\pr^1$ is another section with
 $\wi{G}_Y\cap \mathbb{F}\equiv C_0+a f$,
where $f$ is a fiber of $\pi_{|\mathbb{F}}$ (see Figure \ref{fig2}). 

Since $A\subset\pr^{n-1}$ is a hypersurface of degree $d$, and
$\ell\not\subset A$,  $A\cap \ell$  is a zero-dimensional scheme of length $d$. Moreover 
$A_Y\cap \mathbb{F}$ is isomorphic to $A\cap \ell$, because $\wi{G}_Y$ is a section of $\pi$.
 In particular $A_Y\cap \mathbb{F}$ can be seen as a closed subscheme of $\ell\cong\pr^1$, hence it is determined by its support and by its multiplicity at each point.

Recall that $1\le a\le d$ by assumption. Let us consider a closed subscheme $W$ of $A_Y\cap \mathbb{F}$, of length $a$. Again, $W$ is determined by its support and by its multiplicity at each point, so that without loss of generality we can write $W=\{p_1,\ldots,p_a\}$, where $p_i$ are
 possibly equal points in $A_Y \cap \mathbb{F}$ (and each $p_i$ appears at most $h_i$ times, if $h_i$ is the multiplicity of $A_Y\cap \mathbb{F}$ at $p_i$).
\begin{construction}\label{construction}
We associate to $(\ell,W,x)$ a smooth rational curve $C\subset X$, of anticanonical degree $n$, and containing $x$.
\end{construction}
Set $y:=\sigma(x)\in \mathbb{F}\smallsetminus (\pi^{-1}(A)\cup G_Y \cup \wi{G}_Y)$. 
We write $\mathscr{I}_{W}$ (respectively, 
$\mathscr{I}_{y}$) for the ideal sheaf of $W$ (respectively, $y$) in $\mathbb{F}$. 
We claim that 
$$h^0\left(\mathbb{F},\mathscr{O}_{\mathbb{F}}(C_0+a f)\otimes\mathscr{I}_{W}
\otimes\mathscr{I}_y\right)=1,$$
and that the corresponding curve on $\mathbb{F}$ is smooth and irreducible. 

Since
$h^0(\mathbb{F},\mathscr{O}_{\mathbb{F}}(C_0+a f))=a+2$, we must have 
$h^0(\mathbb{F},\mathscr{O}_{\mathbb{F}}(C_0+a f)\otimes\mathscr{I}_{W}
\otimes\mathscr{I}_y)\ge 1$.

Let $C_1\in |C_0+a f|$ be a curve containing $W$ and $y$.
 Observe first that
$C_1$ is irreducible. Otherwise, since $C_1\cdot f=1$, 
there is a unique irreducible component $C'_1$ of $C_1$ 
such that $C'_1\cdot f=1$ and
$C_1\equiv C'_1+r f$ for some 
$r\ge 1$ ($C'_1$ is a section of $\pi_{|\mathbb{F}}$). In particular, $C'_1\in |C_0+b f|$ with $b<a$. By \cite[Corollary V.2.18]{hartshorne}, we have $b=0$ and $C'_1=C_0$.
Thus $C_1=C_0\cup f_1 \cup \cdots \cup f_a$ where the $f_i$'s are possibly equal fibers of $\pi_{|\mathbb{F}}$. 

Notice that $W$ is a subscheme of both $C_1=C_0\cup f_1 \cup \cdots \cup f_a$ and $\wi{G}_Y\cap \mathbb{F}$.
Since
$\wi{G}_Y\cap \mathbb{F}$ is disjoint from $C_0$, we must have $$W=\{p_1,\dotsc,p_a\}\subseteq (f_1\cup\cdots\cup f_a)\cap \wi{G}_Y.$$ On the other hand  
$\wi{G}_Y$ intersects transversally any fiber of $\pi_{|\mathbb{F}}$, thus we get
$\{p_1,\ldots,p_a\}=
\{f_1\cap \wi{G}_Y, \ldots, f_a\cap \wi{G}_Y\}$, and up to renumbering we can assume that $f_i$ is the fiber of $\pi_{|\mathbb{F}}$ containing $p_i$.
This implies that $y \in C_1\subseteq\pi^{-1}(A)\cup G_Y$, which
 contradicts our choice of $y$. Thus $C_1$ is irreducible,
hence it is a section of $\pi_{|\mathbb{F}}$, and $C_1\cong\pr^1$.

To show that $C_1$ is unique, let 
 $C_2\in |C_0+a f|$ be another curve containing $W$ and $y$. Then $C_2$ is irreducible, $C_1\cdot C_2=a$, and 
$\{p_1,\ldots,p_a,y\} \subseteq  C_1\cap C_2$, which implies that $C_1=C_2$.
This shows our claim.

\stepcounter{remarktwo}
\begin{figure} 
\begin{center}
\scalebox{0.4}{\includegraphics{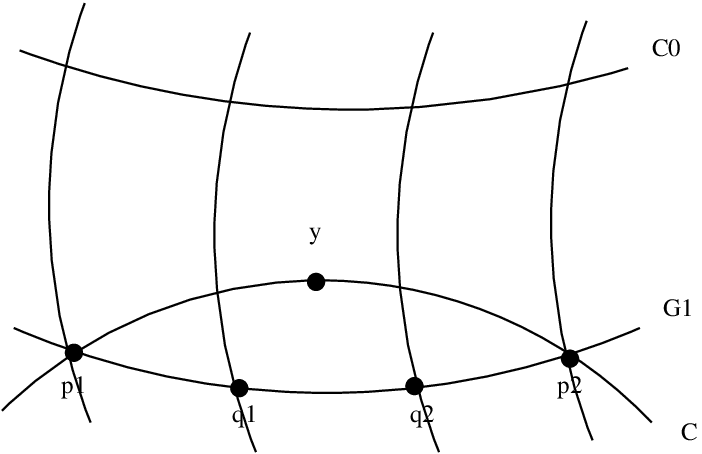}}
\caption{The surface $\mathbb{F}$. Here $d=4$, $a=2$, $A_Y\cap \mathbb{F}=\{p_1,p_2,q_1,q_2\}$, and
 $W=\{p_1,p_2\}$.}\label{fig2} 
\end{center}\end{figure}

\medskip

We remark that:
\renewcommand{\theequation}{\theremarktwo}\stepcounter{remarktwo}
\begin{equation}\label{Cartier} 
A_Y\cap C_1=\wi{G}_{Y|C_1}.\end{equation}
Indeed $C_1$ is not contained in $\wi{G}_Y$, because $y\not\in\wi{G}_Y$. Moreover
 $W\subseteq A_Y\cap C_1\subseteq \wi{G}_Y\cap C_1$ and $\wi{G}_Y\cdot C_1=a$, so that
$W=A_Y\cap C_1=\wi{G}_{Y|C_1}$.

\medskip

We define $C\subset X$ to be the  strict transform of $C_1\subset Y$, so that $C$ is a smooth rational curve through $x$.
 It is not difficult to see that $C\equiv C_G+a\wi{F}$, and hence $-K_X\cdot C=n$ (see Table \ref{table}).
\begin{lemma}\label{free}
If $\ell\cap A$ is either reduced, or has a unique non-reduced point of multiplicity $2$, then:
$$
{T_X}_{|C}\cong \mathscr{O}_{\mathbb{P}^1}(2)\oplus \mathscr{O}_{\mathbb{P}^1}(1)^{n-2}\oplus\mathscr{O}_{\mathbb{P}^1},
$$
hence $C$ is a standard, free, smooth  rational curve.
\end{lemma}
\begin{proof}
Suppose first that $C\cap (E\cap \wi{E})=\emptyset$.
This implies that $\ph$ is smooth in a neighbourhood of $C$, thus
the map ${T_X}_{|C}\to {\ph^* T_{\pr^{n-1}}}_{|C}$ is onto. Its kernel is a torsion free sheaf of rank 1 on $C$, hence a locally free sheaf, of degree $-K_X \cdot C + \ph^*K_{\pr^{n-1}} \cdot C=0$. Since 
${\ph^* T_{\pr^{n-1}}}_{|C}\cong \mathscr{O}_{\mathbb{P}^1}(2)\oplus\mathscr{O}_{\mathbb{P}^1}(1)^{n-2}$, this implies the statement.

Suppose now that $C\cap (E\cap \wi{E})\neq\emptyset$, and let $x_1$ be a point in this intersection. Set $z_1:=\ph(x_1)$, and let $h\in\{1,2\}$ (respectively, $k\in\{1,h\}$) be the multiplicity of $z_1$ in $\ell\cap A$ (respectively, in $W\subseteq\ell\cap A$ under the isomorphism $\wi{G}_Y\cong \pr^{n-1}$).

Notice that $E$ and $\wi{E}$ are Cartier divisors on $X$, and they do not contain $C$ in their support.

As $\sigma$ is an isomorphism on $C$, we have $E_{|C}=C\cap E\cong C_1\cap A_Y= W$.
On the other hand:
$$(E+\wi{E})_{|C}=\ph^*(A)_{|C}=\ph_{|C}^*(A_{|\ell}),$$
therefore we have the following equalities of divisors on $C$:
$$C\cap\widehat{E}=\wi{E}_{|C}=\ph_{|C}^*(A_{|\ell}-W).$$
In particular $z_1$ has multiplicity $h-k$ in $\ph_*(C\cap\widehat{E})$, and since $x_1\in C\cap\widehat{E}$, we have $h-k>0$. We deduce that $h=2$, $k=1$, and $z_1$ is the unique non-reduced point in $\ell\cap A$.

As $W\cap (A_{|\ell}-W)=\{z_1\}$, we also deduce that $C$ intersects 
$E\cap\wi{E}$ only in $x_1$. Moreover, since $C$ is transverse to $E$ (and to $\wi{E}$) in $x_1$, we have:
$$C\cap E\cap\widehat{E}=\{x_1\}.$$

Now let us consider the tangent morphism
 $d\ph\colon {T_X}\to \ph^* T_{\pr^{n-1}}$, which is surjective outside $E\cap\widehat{E}$, and has rank $n-2$ in every point $x\in E\cap\widehat{E}$, with image $T_{A,\ph(x)}$.
Let us also consider the natural surjective morphism:
$$\xi\colon \ph^* T_{\pr^{n-1}}\longrightarrow \ph^*(T_{\pr^{n-1}}/T_A)_{|E\cap\wi{E}}\,.$$
Since $E$ and $\wi{E}$ intersect transversally, a local computation shows 
 that the image 
 of $d\ph$ is the subsheaf of 
$\ph^* T_{\pr^{n-1}}$ given by the kernel of $\xi$. In other words, we have an exact sequence:
$${T_X}\stackrel{d\ph}{\longrightarrow} \ph^* T_{\pr^{n-1}}\stackrel{\xi}{\longrightarrow} \ph^*(T_{\pr^{n-1}}/T_A)_{|E\cap\wi{E}}
\longrightarrow 0.$$

Let us now restrict to the curve $C$. Then
 $\ph^* (T_{\pr^{n-1}}/T_A)_{|E\cap\wi{E}\cap C}=\mathbb{C}_{x_1}$ is a skyscraper sheaf, and we have an exact sequence:
$${T_X}_{|C}\stackrel{d\ph_{|C}}{\longrightarrow} 
\ph^*{T_{\pr^{n-1}}}_{|C}\stackrel{\xi_{|C}}{\longrightarrow} \mathbb{C}_{x_1}
\longrightarrow 0,$$
where
$\xi_{|C}$ is just the evaluation map.

We have $\ph^*{T_{\pr^{n-1}}}_{|C}\cong\mathscr{O}_{\mathbb{P}^1}(2)\oplus\mathscr{O}_{\mathbb{P}^1}(1)^{n-2}$, and
the factor $\mathscr{O}_{\pr^1}(2)\cong\ph^*{T_{\ell}}_{|C}\subset\ph^*{T_{\pr^{n-1}}}_{|C}$ is contained in $\ker(\xi_{|C})$, because $T_{\ell,z_1}\subset T_{A,z_1}$. Therefore we get an induced surjective morphism $\mathscr{O}_{\mathbb{P}^1}(1)^{n-2}\to \mathbb{C}_{x_1}$, whose kernel is $\mathscr{O}_{\pr^1}(1)^{n-3}\oplus\mathscr{O}_{\pr^1}$. This gives an exact sequence
$$0\longrightarrow \mathscr{O}_{\pr^1}(2)\longrightarrow\ker(\xi_{|C})
\longrightarrow \mathscr{O}_{\pr^1}(1)^{n-3}\oplus\mathscr{O}_{\pr^1}
\longrightarrow 0,$$
which yields
$\ker(\xi_{|C})\cong\mathscr{O}_{\pr^1}(2)
\oplus\mathscr{O}_{\pr^1}(1)^{n-3}\oplus\mathscr{O}_{\pr^1}$.

Thus $d\ph_{|C}$ yields a surjective map ${T_X}_{|C}\to\mathscr{O}_{\pr^1}(2)\oplus \mathscr{O}_{\pr^1}(1)^{n-3}\oplus\mathscr{O}_{\pr^1}$.
The kernel of this morphism is a torsion free sheaf of rank 1 on
$C$, hence a locally free sheaf, of degree $-K_X\cdot C-(n-1)=1$. Finally,
the exact sequence
$$0\longrightarrow\mathscr{O}_{\pr^1}(1)\longrightarrow
{T_X}_{|C}\stackrel{d\ph_{|C}}{\longrightarrow}\mathscr{O}_{\pr^1}(2)\oplus \mathscr{O}_{\pr^1}(1)^{n-3}\oplus\mathscr{O}_{\pr^1}
\longrightarrow 0$$
gives the statement.
\end{proof}
\begin{lemma}\label{lu}
Let $C'$ be an effective one-cycle in $X$ such that $C'\equiv C$ and $C'\cap X_0\neq\emptyset$. 

Then $C'$ is integral and is obtained as in \ref{construction}, for some choices of $\ell'\subset\pr^{n-1}$, $x'\in X$, and $W'\subseteq \ell'\cap A$. 
In particular, $C'$ is again a smooth, connected rational curve.
\end{lemma}
\begin{proof}
Since  $\ph_*(C')\equiv \ph_*(C)$, it follows that $\ph(C')$ is a line $\ell'$ in $\pr^{n-1}$, and $\ell'\not\subseteq A$ because $C'\not\subseteq\ph^{-1}(A)=E\cup\wi{E}$. Moreover,  there is a unique irreducible component 
$C''$ of $C'$ that maps onto $\ell'$, and
$C'' \to \ell'$ is a birational morphism.

Therefore we have 
$$C'=C''+ F_1+\cdots+ F_r+\wi{F}_1+\cdots +\wi{F}_s+ e_1+\cdots+ e_h,$$
where the $F_i$'s are (possibly equal) fibers of $\ph_{|E}$, the $\wi{F}_i$'s are fibers of $\ph_{|\wi{E}}$, and the $e_i$'s are fibers of $\ph$ over $\pr^{n-1}\smallsetminus A$. Moreover $C'\equiv C_G+a\wi{F}$,  and using Table \ref{table} we get: $$0=G\cdot C'=G\cdot C''+s+h\quad\text{and}\quad
0=\wi{G}\cdot C'=\wi{G}\cdot C''+r+h.$$
Since $C''$ is irreducible and $G\cap\wi{G}=\emptyset$, the intersections $G\cdot C''$ and $\wi{G}\cdot C''$ cannot  be both negative; this yields $h=0$. Then $C''$ cannot be contained in $G\cup\wi{G}$, because $C'$ is not contained in $E\cup\wi{E}\cup G \cup \wi{G}$. We conclude that $r=s=0$, and $C'=C''$ is integral.

Set $\mathbb{F}':=\pi^{-1}(\ell')\subset Y$, and denote by $\pi_{|\mathbb{F}'}$ the restriction of $\pi$ to $\mathbb{F}'$. We consider the curve $C_1':=\sigma(C')\subset \mathbb{F}'$.

Since $\pi_{|C'_1}$ has degree one, we have $C'_1\equiv C'_{0}+rf'$, where
$C'_{0}=G_Y\cap \mathbb{F}'$ is the minimal section of $\pi_{\ell'}$, $f'$ is one of its fibers, and $r\in\mathbb{Z}$.

On the other hand we have
$$-K_{\mathbb{F}'/\ell'}\cdot C'_1= -K_{Y/\pr^{n-1}}\cdot C'_1=-K_{Y/\pr^{n-1}}\cdot C_1=-K_{\mathbb{F}/\ell}\cdot C_1=a,$$
where $C_1:=\sigma(C)$.
Therefore $r=a$, and 
 $C'_1 \in |C'_{0}+af'|$.

Finally, as $E\cdot C'=a$ and $C'_1$ is smooth, $C'_1$ intersects $A_Y$ in a zero-dimensional subscheme $W'$ of length $a$. This shows that $C'$ is a curve obtained via Construction \ref{construction}.
\end{proof}
\begin{proof}[Proof of Proposition~\ref{examples}]
If $a=0$, the statement is clear.

Suppose that $a>0$. By Theorem \ref{grass}, pairs $(\ell,W)$ where $\ell\subset\pr^{n-1}$ is a line not contained in $A$, and  $W\subseteq \ell\cap A$ is a subscheme of length $a$, vary in an irreducible family of dimension $2n-4$.
 Using this, it is not difficult to show that varying $(\ell, W,x)$,  construction \ref{construction} yields an irreducible algebraic family of smooth, connected rational curves in $X$, of anticanonical degree $n$. 
More precisely, we get a locally closed irreducible subvariety $V_0$ of $\Rat(X)$ of dimension $2n-3$, whose points correspond to curves $C$ obtained as in \ref{construction}.

By Lemma \ref{free}, a general curve $C$ from $V_0$ is free and hence yields a smooth point in $\Rat(X)$. Moreover,
$\dim_{[C]} \Rat(X)=-K_X\cdot [V]+n-3=2n-3$.
This implies that the closure $V$ of $V_0$ 
in $\Rat(X)$ is an irreducible component, and that $V$ gives a dominating family of curves of anticanonical degree $n$.

On the other hand, Lemma \ref{lu} shows that $V_x$ is proper for every $x\in X_0$, so that $V$ is a locally unsplit family.
\end{proof}
\begin{remark}\label{minimaldegree}
Let $\mathscr{H}$ be an ample line bundle on $X$. If $2\le a \le d-2$, then $V$ is not a dominating family of rational curves of minimal degree with respect to $\mathscr{H}$. 
Indeed, the family $W$ of curves on $X$ whose points correspond to smooth fibers of $\ph$ is a locally unsplit dominating family of rational curves, and $[W]\equiv F+\wi{F}$. Thus   
$$\mathscr{H}\cdot [V]=\frac{1}{2}\left(\mathscr{H}\cdot C_G+\mathscr{H}\cdot C_{\wi{G}}+
a \mathscr{H}\cdot \wi{F}+(d-a) \mathscr{H}\cdot F\right)
>\mathscr{H}\cdot [W].$$
\end{remark}
\begin{proof}[Proof of Theorem~\ref{main}]
The first part of the statement follows from  Proposition~\ref{prop_main}. Assume that $\rho_X=3$. Then, 
again by Proposition~\ref{prop_main},
 $X$ is isomorphic to one of the varieties described in Example~\ref{ex}, and 
$[V]\equiv C_{\wi{G}}+(d-a)F$ (notation as in Example~\ref{ex} and Proposition~\ref{prop_main}). By \eqref{equiv}, this is the same as $[V]\equiv C_G+a\wi{F}$. This means that the curves of the family $V$ are numerically equivalent to the curves obtained in Construction \ref{construction}. Thus Lemma \ref{lu} implies that the family $V$ coincides with the family of curves constructed in the proof of Proposition \ref{examples}.
\end{proof}
\begin{proof}[Proof of Theorem~\ref{V_x}]
As before, we can assume that $a>0$.
Let $x\in X_0$ be a general point, and set  $z:=\ph(x)\in\pr^{n-1}$.
Let $p_z\colon A \to \mathbb{P}^{n-2}$ be the morphism of degree $d$ induced by the linear
 projection $\mathbb{P}^{n-1}\dashrightarrow \mathbb{P}^{n-2}$ from $z$, where we see $\pr^{n-2}$ as the variety of lines $\ell$ through $z$ in $\pr^{n-1}$. Then the pairs  $(\ell,W)$ such that $\ell\subset\pr^{n-1}$ is a line through $z$, and $W\subset A\cap\ell$ is a zero-dimensional subscheme of length $a$, are parametrized by the relative Hilbert scheme $\mathcal{H}_z:=\textup{Hilb}^{[a]}(A/\mathbb{P}^{n-2})$.

Since $x$ is general, $\mathcal{H}_z$
is integral by Theorem~\ref{genus_hilbert}. 
Therefore, using \cite[Corollary III.12.9]{hartshorne}, Construction  \ref{construction} can be made  relatively over $\mathcal{H}_{z}$. 

One first constructs a subscheme $\mathcal{C}_1\subset Y\times \mathcal{H}_{z}$, such that the fiber of $\mathcal{C}_1\to \mathcal{H}_{z}$ over $(\ell,W)$ is the curve $C_1$. As $\mathcal{C}_1$ intersects $A_Y\times \mathcal{H}_{z}$ along the Cartier divisor $\wi{G}_Y\times \mathcal{H}_{z}$ (see \eqref{Cartier}), the  strict transform $$\mathcal{C}\subset X\times \mathcal{H}_{z}$$ of  $\mathcal{C}_1\subset Y\times \mathcal{H}_{z}$ is isomorphic to $\mathcal{C}_1$. In particular, the fiber  of $\xi\colon\mathcal{C}\to \mathcal{H}_{z}$ over $(\ell,W)$ is the curve $C$.

As $C$ is a smooth rational curve through $x$, $\xi\colon\mathcal{C}\to \mathcal{H}_{z}$ is a $\pr^1$-bundle with a section $s\colon \mathcal{H}_{z}\to \mathcal{C}$, given by $s(h)=(x,h)$. 

In particular, $\mathcal{C}$ is the projectivization of a rank $2$ vector bundle over $\mathcal{H}_{z}$, and it is locally trivial in the Zariski topology.

Fix a point $0\in\pr^1$.
Let $H_0\subseteq \mathcal{H}_{z}$ be an open subset such that $\xi^{-1}(H_0)\cong \pr^1\times H_0$. We can assume that the section $s_{|H_0}$, under this isomorphism, is identified with the constant section $\{0\}\times H_0$.

Therefore we get a morphism over $H_0$:
$$\pr^1\times H_0\to X\times H_0,$$
which is an embedding on $\pr^1\times\{h\}$, and sends $(0,h)$ to $(x,h)$, for every $h\in H_0$.
This yields a morphism $H_0\to\Hom(\pr^1,X,0\mapsto x)$.

Since $x$ is general, 
 every curve in $V_x$ is free, and 
$\Hom(\pr^1,X,0\mapsto x)$ contains a union of smooth irreducible components $\widehat{V}_x$, whose image in $\Rat(X,x)$ is $V_x$.
By construction, the morphism  $H_0\to\Hom(\pr^1,X,0\mapsto x)$  takes values in $\widehat{V}_x$.

By \cite[Theorem II.2.16]{kollar}, these morphisms glue together and yield a morphism
$$\Psi\colon \mathcal{H}_{z}\longrightarrow V_x.$$ 

By Lemma \ref{lu} the morphism $\Psi$ is surjective. Moreover, the pair $(\ell,W)$ determines uniquely the curve $C$ in Construction \ref{construction}. As $C$ is smooth, it corresponds to a unique point in $V_x$; therefore
 $\Psi$ is injective.

Finally, $V_x$ being smooth, we conclude that $\Psi$ is an isomorphism, $\mathcal{H}_{z}$ is smooth, and $V_x$ is irreducible.
\end{proof}
Notice that $a$ and $a'=d-a$ yield not only the same variety $X$ (see Remark \ref{symm}), but also the same family $V$ of rational curves on $X$, and $V_x\cong\Hilb^{[a]}(A/\mathbb{P}^{n-2})\cong\Hilb^{[d-a]}(A/\mathbb{P}^{n-2})$ for general $x\in X$.

The cases $a\in\{0,d\}$ and $a\in\{1,d-1\}$ are the simplest ones. For the reader's convenience, we describe explicitly the latter. 
\begin{example}\label{d=a+1} If $a=1$, then
 ${Y}$ is the blow-up of $\pr^n$ at a point. Thus $X$ is the blow-up of $\pr^n$ along $\{p_0\}\cup A$, where $A$ is smooth, of dimension $n-2$, degree $d$, contained in a hyperplane $H$, and $p_0\not\in H$.
This  is one of the few examples of Fano manifolds obtained by blowing-up a point in another manifold, see 
 \cite{bonwisncamp}.

The general curve of the family $V$ is the  strict transform in $X$ of a line in $\pr^n$ intersecting $A$ in one point.

Therefore for a general point $x\in X$ we have
$V_x\cong A\cong\Hilb^{[1]}(A/\mathbb{P}^{n-2})$.

From the point of view of the family of curves, this is essentially the same example as Example \ref{exampleM}; see also \cite[Example 1.7]{hwang_equivalence}.
\end{example}
Let us consider now the morphism
$$\tau_x\colon V_x \longrightarrow \mathcal{C}_x\subset\mathbb{P}(T_{X,x}^*),$$ 
where $x$ is a general point (notation as in the Introduction).
We recall the following useful observation.
\begin{remark}\label{standard}
Let $[C]\in V_x$. Then
$\tau_x$ is an immersion at $[C]$ if and only if $C$ is standard, see \cite[Proposition 1.4]{hwangICTP} and \cite[Proposition 2.7]{carolina06}. 
\end{remark}
As $x$ is general, $\ph\colon X\to\pr^{n-1}$ is smooth at $x$, and the differential of $\ph$ at $x$ induces the linear projection $\mathbb{P}(T_{X,x}^*) \dashrightarrow \mathbb{P}(T_{\pr^{n-1},z}^*)$
from the point
$[T_{X/\pr^{n-1},x}]\in \mathbb{P}(T_{X,x}^*)$.

We have $[T_{X/\pr^{n-1},x}]\not\in\mathcal{C}_x$, because $\ph_{|C}$ is an isomorphism for every $C$ in $V_x$, and the projection restricts to a morphism $\Pi'\colon \mathcal{C}_x\to \pr^{n-2}=\mathbb{P}(T_{\pr^{n-1},z}^*)$.

Observe that there is a commutative diagram:
$$\xymatrix{{\mathcal{H}_z}\ar@/^1pc/[rr]^{\Phi}\ar[r]_{\Psi}\ar[dr]_{\Pi}& 
{V_x} \ar[r]_{\tau_x}&  {\mathcal{C}_x}\ar[dl]^{\Pi'}\\
 &  {\pr^{n-2}} &
}
$$
where we have set $\Phi:=\tau_x\circ\Psi\colon \mathcal{H}_z\to\mathcal{C}_x$.
\begin{proof}[Proof of Theorem \ref{tau_x}]
Since $\Pi$ has degree $\binom{d}{a}$, $\Psi$ is an isomorphism, and $\tau_x$ is birational, we conclude that $\Pi'$ has degree   $\binom{d}{a}$, therefore
 $\mathcal{C}_x\subset\pr^{n-1}$ is a hypersurface of degree $\binom{d}{a}$. Moreover $\mathcal{C}_x$ is irreducible, because $V_x$ is.

We have seen in the above Example
that $\tau_x$ is an isomorphism for $a=1$ and $a=d-1$, and the statement is clear if $a=0$ or $a=d$.

We suppose from now on that $2\leq a\leq d-2$; notice that in particular $d\geq 4$.
Since $\Psi$ is an isomorphism, the statement follows if we show that the closed subset where $\Phi$ is not an isomorphism (respectively, an immersion) has codimension $1$ (respectively, $2$).

Let $L\subset\pr^{n-2}$ be a general line. Then $A_L:=p_z^{-1}(L)$ is a smooth plane curve of degree $d$,
and
$\Pi^{-1}(L)=\Hilb^{[a]}(A_L/L)$ is a smooth curve of genus $1+\frac{1}{2}\binom{d}{a}(a(d-a)-2)$ by Theorem \ref{genus_hilbert} (4). 

On the other hand,
$(\Pi')^{-1}(L)$ is a plane curve of degree $\binom{d}{a}$, hence
it has arithmetic genus 
$p_a=\frac{1}{2}(\binom{d}{a}-1)(\binom{d}{a}-2)$. 

We claim that $\Phi_{|\Pi^{-1}(L)}$ cannot be an isomorphism onto its image.  By contradiction, if it were, the two curves should have the same genus, so we get:
$$\frac{(\binom{d}{a}-1)(\binom{d}{a}-2)}{2}=
1+\frac{1}{2}\binom{d}{a}\left(a(d-a)-2\right).$$
This yields easily
$\binom{d}{a}=a(d-a)+1$,
or equivalently $d(d-1)\cdots (d-a+1)=a!(a(d-a)+1)$.
 Since $a \le d-2$, we must have  
$$(d-2)\cdots(d-a+1)\geq \frac{a!}{2},$$
and hence 
$$d(d-1)=\frac{a!(a(d-a)+1)}{(d-2)\cdots(d-a+1)}\le 2(a(d-a)+1).$$ Equivalently, we get $d^2-(2a+1)d+2a^2-2\le 0$. It is easy to see that this contradicts the assumption $2\leq a\leq d-2$.

We conclude that the closed subset where $\Phi$ is not an isomorphism has codimension $1$.

\medskip

Let us consider again the plane curve $A_L$. Since  $A_L\to L$ is a general projection, every non-reduced fiber contains just one double point. Hence,
for every $[\ell]\in L$, the intersection $\ell\cap A$ is either reduced, or has  a unique non-reduced point, of multiplicity $2$. By Lemma \ref{free}, for every $W\subseteq \ell\cap A$,
the corresponding curve $C\subset X$ is standard, therefore $\tau_x$ is an immersion at $\Psi([W])$ by Remark \ref{standard}.

We have shown that  $\Phi$ is an immersion in every point of $\Pi^{-1}(L)$, hence the
closed subset where $\Phi$ is not an immersion has codimension at least $2$.
 
\medskip

Suppose now that $n\geq 4$, and let $P\subset\pr^{n-2}$ be a general plane. Then $A_P:=p_z^{-1}(P)$ is a smooth surface of degree $d$ in $\pr^3$, and $A_P\to P$ is the projection from a general point. Let $B\subset\pr^2$ be the branch curve of this projection. It is classically known that $B$ has only nodes and cusps, see \cite{cilibertoflamini}. Moreover, the fiber of $A_P\to P$ over a smooth point $z\in B$ contains just one double point, the fiber over a node contains two double points, and the fiber over a cusp contains one triple point (see \cite[Proposition 3.7]{cilibertoflamini}). The number of nodes in $B$ depends only on the degree $d$ of $A_P$, and more precisely it is $d(d-1)(d-2)(d-3)/2$, see \cite[Lemma 3.2(a)]{FLS}. 

Since $d\geq 4$, $B$ contains nodes, and we conclude that 
 there is at least one $[\ell]\in P$ such that $A\cap\ell=2p_1+2p_2+p_3+\cdots+p_{d-2}$. 
As $a\leq d-2$, we can consider the subscheme
 $W:=p_1+p_2+\cdots+p_a$ and the point $[W]\in \mathcal{H}_z$. 
Notice that the points $p_1$ and $p_2$ do appear in $W$, because $a\geq 2$.
Then by Theorem \ref{genus_hilbert} (2), the tangent space of the fiber $\Pi^{-1}([\ell])$ at $[W]$ has dimension $2$. 

On the other hand, as $\Pi'$ is a linear projection from a point, the tangent space of the fiber   $(\Pi')^{-1}([\ell])$ at $\Phi([W])$ has dimension at most $1$. This shows that $\Phi$ is not an immersion at $[W]$, and hence that the closed subset where $\Phi$ is not an immersion has codimension $2$.
\end{proof}
\begin{corollary}
Let $X$ and $V$ be as in Proposition \ref{examples}. Assume that $2\leq a\leq d-2$, and that $n\geq 4$. Then non-standard curves of the family $V$ cover $X$.
\end{corollary}
\begin{proof}
This follows from 
Remark \ref{standard} and Theorem \ref{tau_x}.
\end{proof}
\begin{example}[Fano $3$-folds]
Let $X$ and $V$ be as in Proposition \ref{examples}, and consider the case $n=3$. Then $X$ is Fano if and only if $a\leq 2$ and $d-a\leq 2$, in particular $d\leq 4$. Thus the only case where $X$ is Fano and $\tau_x$ is not an isomorphism is for $d=4$ and $a=2$. This Fano $3$-fold $X$ is N.~9 in \cite[\S 12.4]{FanoEMS}.

In this case $A\subset \pr^2$ is a smooth quartic,  $V_x\cong \mathcal{H}_z=\Hilb^{[2]}(A/\pr^1)$ is a smooth connected curve of genus $7$, and $\Pi\colon \mathcal{H}_z\to\pr^1$ has degree $6$. Using Theorem \ref{genus_hilbert} (2), one can describe precisely the ramification of $\Pi$.

On the other hand, $\mathcal{C}_x\subset\pr^2$ is an irreducible curve of degree $6$ and arithmetic genus $10$. The normalization $\mathcal{H}_z\to\mathcal{C}_x$ is an immersion, but it is not injective.
\end{example}
\section{Appendix: the relative Hilbert scheme}\label{appendix}
The proof of Theorem \ref{V_x} relies on the following results of independent interest.
\begin{thm}\label{genus_hilbert}
Fix integers $m$, $a$, and $d$, such that 
$m\geq 1$ and
$1\leq a\leq d$. Let $A\subset \pr^{m+1}$
be a smooth hypersurface of degree $d$,
 $z\in \pr^{m+1}\smallsetminus A$ a general point, and  
$p_z\colon A \to \mathbb{P}^{m}$ the linear projection from $z$ (where we identify $\pr^m$ with the variety of lines through $z$ in $\pr^{m+1}$).
\begin{enumerate}[(1)]
\item The relative Hilbert scheme $\textup{Hilb}^{[a]}(A/\mathbb{P}^{m})$
is an integral local complete intersection scheme of dimension $m$, and the natural morphism 
$$\Pi\colon\textup{Hilb}^{[a]}(A/\mathbb{P}^{m})\to \mathbb{P}^{m}$$
is flat and finite of degree $\binom{d}{a}$.
\item Let $[\ell]\in\pr^m$ and $[W]\in\Pi^{-1}([\ell])$.
Write $\ell\cap A=h_1p_1+\cdots+ h_rp_r$ with  $h_i\geq 1$ and $p_i\neq p_j$ for $i\neq j$, and 
$W=k_1p_1+\cdots+ k_rp_r$
 with $0\leq k_i\leq h_i$. The Zariski tangent space of the fiber $\Pi^{-1}([\ell])$ at $[W]$ has dimension
$\sum_{i=1}^r\min(k_i,h_i-k_i)$.
\item $\Pi$ is smooth at $[W]$ if and only if $W$ is a union of irreducible components of $\ell\cap A$, equivalently: $W\cap(\ell\cap A-W)=\emptyset$. 
\item Suppose that $m=1$. Then the curve $\textup{Hilb}^{[a]}(A/\mathbb{P}^1)$ is smooth of genus 
$g=1+\frac{1}{2}\binom{d}{a}(a(d-a)-2)$.
\end{enumerate}
\end{thm}
The following results will be used in the proof of Theorem \ref{genus_hilbert}.
\begin{lemma}\label{local}
Let $h$ be a positive integer, and set:  
$$\Lambda:=\frac{\mathbb{C}[t]}{(t^{h})}\quad\text{ and }\quad F:=\Spec \Lambda.$$ Let $W$ be the non-empty closed subscheme 
of $F$ with ideal $I:=t^{k}\Lambda\subseteq \Lambda$, where $k\in\{1,\dotsc,h\}$ is an integer. Then:
\begin{enumerate}[(1)]
\item $\dim_{\mathbb{C}} \Hom_{F}(\mathscr{I}_{W},\mathscr{O}_{W}) 
= \dim_{\mathbb{C}} \Ext^1_{F}(\mathscr{I}_{W},\mathscr{O}_{W})=\min(k,h-k)$;
\item
$\Hilb(F)$ has dimension zero,  $\Obs(W)=\Ext^1_{F}(\mathscr{I}_{W},\mathscr{O}_{W})$, and the Zariski tangent space of $\Hilb(F)$ at $[W]$ has dimension $\min(k,h-k)$;
\item
$\Hilb(F)$ is smooth at $[W]$ if and only if $W=F$.
\end{enumerate}
\end{lemma}
\begin{proof}[Proof of Lemma \ref{local}]
We have a short sequence of $\Lambda$-modules:
$$0 \longrightarrow t^{h-k} \Lambda \longrightarrow  \Lambda \longrightarrow
t^{k}\Lambda=I \longrightarrow 0,$$
where the second morphism is given by $1\mapsto t^{k}$, so that $I\cong \Lambda/t^{h-k} \Lambda$ as $\Lambda$-modules. Using this isomorphism, 
a direct computation shows that $\dim_{\mathbb{C}} \Hom_{\Lambda}(I,\Lambda/I)
=\min(k,h-k)$.

Applying 
the functor $\Hom_{\Lambda}(-,\Lambda/I)$
to the above sequence, 
and using the vanishing of
$\Ext^1_{\Lambda}(\Lambda,\Lambda/I)$,
we get the exact sequence of $\Lambda$-modules:
\begin{multline*}
0 \longrightarrow 
\Hom_{\Lambda}(I,\Lambda/I)
\longrightarrow
\Hom_{\Lambda}(\Lambda,\Lambda/I)\\
\longrightarrow
\Hom_{\Lambda}(t^{h-k}\Lambda,\Lambda/I)
\longrightarrow
\Ext^1_{\Lambda}(I, \Lambda/I)
\longrightarrow
0.
\end{multline*}
Similarly as before, observe that 
 $t^{h-k} \Lambda\cong \Lambda/I$ as $\Lambda$-modules. Moreover,  if $\pi\colon \Lambda\to \Lambda/I$ is the quotient map,
it is easy to
see that   $\pi^*\colon \Hom_{\Lambda}(\Lambda/I,\Lambda/I)\to\Hom_{\Lambda}(\Lambda,\Lambda/I)$  is an isomorphism of $\Lambda$-modules. Therefore 
$$\Hom_{\Lambda}(t^{h-k} \Lambda,\Lambda/I)\cong\Hom_{\Lambda}(\Lambda/I,\Lambda/I)\cong \Hom_{\Lambda}(\Lambda,\Lambda/I),$$ and
we obtain (1):
$$
\dim_{\mathbb{C}} \Hom_{\Lambda}(I,\Lambda/I)
= \dim_{\mathbb{C}} \Ext^1_{\Lambda}(I,\Lambda/I)=\min(k,h-k).$$

The Hilbert scheme of $F$ is supported on finitely many points, thus it has dimension zero. Recall that by \cite[Definition I.2.6]{kollar}, the obstruction space $\Obs(W)$ is a subspace of $\Ext^1_{F}(\mathscr{I}_{W},\mathscr{O}_{W})\cong \Ext^1_{\Lambda}(I,\Lambda/I)$. Then (2) follows from (1) and
 \cite[Theorems I.2.8.1 and I.2.8.4]{kollar}.

Finally, (3) follows from (2).
\end{proof}
\begin{lemma}\label{local2}
Let $p\colon A\to T$ be a finite morphism between smooth quasi-projective varieties. 
Let $F$ be a fiber of $p$, and $W\subseteq F$ a reduced subscheme of length $a$.
Let $\Pi\colon\textup{Hilb}^{[a]}(A/T)_{red}\to T$ be the natural morphism.
Suppose that $F$ has a unique non-reduced point at $z_1$, and that $z_1\in \textup{Supp}(W)$.
 
Then the Hilbert scheme $\Hilb^{[a]}(A/T)_{red}$ is smooth at $[W]$, 
and there exists a neighbourhood for the Euclidean topology $U\subset \Hilb^{[a]}(A/T)_{red}$ (respectively, $U_1\subset A$) of $[W]$
(respectively, $z_1$) and an isomorphism 
$\iota\colon U_1\cong U$ (of complex manifolds)
such that the diagram: 
$$\xymatrix{
{U_1}\ar[dr]_{p_{|U_1}}\ar[rr]^{\iota} & & U\ar[dl]^{\Pi_{|U}}\\
& T &
}
$$
commutes.
\end{lemma}
\begin{proof}
Recall that  the 
Hilbert-Chow morphism 
$\textup{Hilb}^{[a]}(A) \to A^{(a)}:=A^a/\mathbb{S}_a$ 
maps a zero-dimensional subscheme of length $a$ of $A$ to the associated 
effective $0$-cycle of degree $a$.
Notice that
$[W]$ is contained in the open subset of $\Hilb^{[a]}(A)$ where the Hilbert-Chow morphism is an isomorphism, and that 
$\Hilb^{[a]}(A)$ is smooth at $[W]$
since $[W]$ is reduced.

Let $V\subset T$ be an open neighbourhood of $p(z_1)$ for the Euclidean topology such that $p^{-1}(V)=U_1\cup\cdots \cup U_r$,
$U_i\cap U_j=\emptyset$ if $i\neq j$, $U_1\subset A$ is an open neighbourhood of $z_1$,
$p_{|U_i}\colon U_i\to V$ is an isomorphism (of complex manifolds) 
for $i\ge 2$, and $W \cap U_i\neq\emptyset$ if and
only if $1 \le i\le a$.

Let us consider the map $f\colon U_1\to A^a$ given by
$$f(z)=\left(z,(p_{|U_2})^{-1}(p(z)),\ldots,(p_{|U_a})^{-1}(p(z))\right).$$
Then $f$ is a holomorphic immersion. Moreover, for every $z\in U_1$, the points 
$z,(p_{|U_2})^{-1}(p(z)),\ldots,(p_{|U_a})^{-1}(p(z))\in A$
are pairwise distinct, so that the composition of $f$ with the quotient map $A^a\to A^{(a)}$ is still an immersion. 
This yields a holomorphic immersion
$$\iota\colon U_1 \longrightarrow \textup{Hilb}^{[a]}(A),$$
such that  $U:=\iota(U_1)\subset \textup{Hilb}^{[a]}(A/T)_{red}$. Moreover, $\iota$ is injective, $\iota(z_1)=[W]$, and
$p_{|U_1}=\Pi_{|U}\circ \iota$. Our claim follows easily.
\end{proof}
\begin{proof}[Proof of Theorem \ref{genus_hilbert}]
Let $\ell\subset\pr^{m+1}$ be a line passing through $z$, and set $F:=\ell\cap A$, so that 
$F$ is a zero-dimensional subscheme of $\ell\smallsetminus \{z\}\cong\mathbb{A}^1$. We have:
 $$\Pi^{-1}([\ell])=\Hilb^{[a]}(F);$$ 
in particular $\Pi$ is a finite morphism, 
and $\dim\textup{Hilb}^{[a]}(A/\mathbb{P}^{m}) \le m$.

Let $[W]\in \textup{Hilb}^{[a]}(A/\mathbb{P}^{m})$ be a point over 
$[\ell] \in \mathbb{P}^{m}$. Applying  Lemma \ref{local} to every connected component of $\Pi^{-1}([\ell])$, we get (2) and (3), and also that
$\dim_{\mathbb{C}} \Hom_{F}(\mathscr{I}_{W},\mathscr{O}_{W}) 
= \dim_{\mathbb{C}} \Obs(W)$.
 Thus, by \cite[Theorems I.2.10.3 and I.2.10.4]{kollar},
any irreducible component of 
$\textup{Hilb}^{[a]}(A/\mathbb{P}^{m})$
through $[W]$ has dimension $m$,
and $\Pi$ is a local complete intersection morphism. In particular, $\textup{Hilb}^{[a]}(A/\mathbb{P}^{m})$
is a local complete intersection scheme, and $\Pi$ is a flat finite morphism. 

By (3),  $\Pi$ is \'etale over $[\ell]$ if and only if $\ell\cap A$ is reduced, \emph{i.e.} $p_z$ is \'etale over $[\ell]$. Therefore 
$\textup{Hilb}^{[a]}(A/\mathbb{P}^{m})$ is generically smooth over $\mathbb{P}^{m}$. In particular, $\textup{Hilb}^{[a]}(A/\mathbb{P}^{m})$ is generically reduced and hence reduced as it is a Cohen-Macaulay scheme.

\medskip

We proceed to show that the scheme $\textup{Hilb}^{[a]}(A/\mathbb{P}^{m})$ is irreducible. This follows from the fact that since $z$ is general, 
 the monodromy group
of the projection $p_z\colon A \to \mathbb{P}^m$ is the whole symmetric group 
$\mathbb{S}_d$, see \cite[Proposition 2.3]{cukierman}.

Let $U \subset \mathbb{P}^m$ be a dense open subset such that $p_z$ and $\Pi$ are \'etale over $U$.  
Set $A_0:=p_z^{-1}(U)\subseteq A$ and $H_0:=\Pi^{-1}(U)=\textup{Hilb}^{[a]}(A_0/U)$. 
Since $H_0$ is dense in $\textup{Hilb}^{[a]}(A/\mathbb{P}^{m})$, and $H_0$ is smooth, we are reduced to show that $H_0$ is connected. 

Notice that $H_0$ is a closed subscheme of $\Hilb^{[a]}(A_0)$, and that every $[W]$ in $H_0$ is a reduced subscheme of $A_0$, so that 
$H_0$ is contained in the open subset of  $\Hilb^{[a]}(A_0)$ where the Hilbert-Chow morphism 
$\textup{Hilb}^{[a]}(A_0) \to (A_0)^{(a)}$
is an isomorphism.

Let 
$[W_1]$ and $[W_2]$ be two points in $H_0$ that map to a given point in $U$. Since any irreducible component of $H_0$ maps onto $U$, it is enough to prove that 
there is a path in $H_0$ joining $[W_1]$ to $[W_2]$.

Let $[\wi{W}_i]\in (A_0)^a$ mapping to $[W_i]$. Here, we are considering the natural map 
$(A_0)^a \to (A_0)^{(a)}$, composed with the inverse of the Hilbert-Chow morphism. 
Since 
the monodromy group
of $p_z$ is the symmetric group 
$\mathbb{S}_d$, there is a
path $\gamma\colon [0,1]\to (A_0)^a$ joining $[\wi{W}_1]$ to $[\wi{W}_2]$, such that the points in $\gamma(t)$ are distinct and contained in a fiber of $p_z$, for every $t\in[0,1]$ .
This yields a path joining
$[W_1]$ to $[W_2]$ in $H_0$, and proves that   
$\textup{Hilb}^{[a]}(A/\mathbb{P}^{m})$ is integral.

\medskip

Finally, suppose that $m=1$. 
We show that $\textup{Hilb}^{[a]}(A/\mathbb{P}^1)$ is a smooth curve of genus 
$$g=1+\frac{1}{2}\binom{d}{a}\left(a(d-a)-2\right).$$

Notice that since the projection $p_z\colon A\to\pr^1$ is general, 
there are precisely $d(d-1)$ 
non-reduced fibers, and
every non-reduced fiber contains just one non-reduced point, with multiplicity 2. 

Let $[W]\in \textup{Hilb}^{[a]}(A/\mathbb{P}^1)$ be a point over
$[\ell]\in\pr^1$. By (3), either
$\Pi$ is \'etale at $[W]$, or
$\ell\cap A$ has a double point, 
$W$ is reduced, and contains this point in its support. 
Note that there are exactly $\binom{d-2}{a-1}$ such $[W]$'s, and that $\Pi$ has ramification index 2 at any of these points by Lemma \ref{local2}.

If $\Pi$ is \'etale at $[W]$, then $\textup{Hilb}^{[a]}(A/\mathbb{P}^1)$ is obviously smooth at $[W]$. Otherwise, $\textup{Hilb}^{[a]}(A/\mathbb{P}^1)$ is  smooth at $[W]$ by Lemma \ref{local2}. This shows that $\textup{Hilb}^{[a]}(A/\mathbb{P}^1)$ is a smooth curve.

By the Hurwitz formula, we have:
$$2g-2=-2\binom{d}{a}+d(d-1)\binom{d-2}{a-1}=\binom{d}{a}(a(d-a)-2).$$
This completes the proof of the theorem.
\end{proof}
\begin{remark}
The same proof shows that for \emph{every} $z\in\pr^{m+1}\smallsetminus A$, the  relative Hilbert scheme $\textup{Hilb}^{[a]}(A/\mathbb{P}^{m})$
is a reduced local complete intersection scheme of dimension $m$, and that the natural morphism 
$\Pi\colon\textup{Hilb}^{[a]}(A/\mathbb{P}^{m})\to \mathbb{P}^{m}$
is flat and finite of degree $\binom{d}{a}$. Moreover, (2) and (3) hold true.
\end{remark}
We need also a slightly more general version of the previous construction, as follows. 
Let $A\subset\pr^{m+1}$ be a smooth hypersurface of degree $d\geq 1$, and fix $a\in\{1,\dotsc,d\}$.
Set 
$$\mathcal{G}:=\left\{[\ell]\in G(1,m+1)\,|\,\ell\text{ is not contained in $A$}\right\},$$ with its universal family
$\mathcal{U}:=\{([\ell],z)\in\mathcal{G}\times\pr^{m+1}\,|\,z\in\ell\}$. 
Let us consider the intersection:
$$\mathcal{I}:=\mathcal{U}\cap (\mathcal{G}\times A).$$
The induced morphism $\mathcal{I}\to\mathcal{G}$ is finite and flat, of degree $d$; the fiber over a line $[\ell]$ is $\ell\cap A$. The relative Hilbert scheme  
$\textup{Hilb}^{[a]}(\mathcal{I}/\mathcal{G})$
parametrizes pairs $(\ell,W)$ where $\ell\subset\pr^{m+1}$ is a line not contained in $A$, and $W$ is a subscheme of length $a$ of $\ell\cap A$.
\begin{thm}\label{grass}
The Hilbert scheme $\textup{Hilb}^{[a]}(\mathcal{I}/\mathcal{G})$ is an integral scheme of dimension $2m$.
\end{thm}
\begin{proof}
The proof is very similar to that of Theorem \ref{genus_hilbert}, 
and so we leave some details to the reader. First, one shows that 
$\textup{Hilb}^{[a]}(\mathcal{I}/\mathcal{G})$ is a reduced local complete intersection scheme
of dimension $2m$ equipped with a finite flat morphism   
$\textup{Hilb}^{[a]}(\mathcal{I}/\mathcal{G}) \to \mathcal{G}$. 

We remark that if $z\in\pr^{m+1}$ is a general point, and $P:=\{[\ell]\in\mathcal{G}\,|\,z\in\ell\}$, then $P\cong\pr^m$, and the inverse image of $P$ in $\textup{Hilb}^{[a]}(\mathcal{I}/\mathcal{G})$ is the relative Hilbert scheme $\textup{Hilb}^{[a]}(A/\mathbb{P}^{m})$ of the projection $p_z\colon A\to\pr^m$ from $z$. 
Then, the same argument used in the proof of Theorem \ref{genus_hilbert} shows that
$\textup{Hilb}^{[a]}(\mathcal{I}/\mathcal{G})$ is irreducible. This completes the proof of the theorem.
\end{proof}
\begin{proof}[Proof of Theorem \ref{app}]
Let $X$ be as in Example \ref{equiv}, with $n=m+2$. Then the statement follows from Proposition \ref{examples}, Theorem \ref{main}, and Theorem \ref{genus_hilbert}.
\end{proof}

\begin{thebibliography}{BCHM10}

\bibitem[Ara06]{carolina06}
C.~Araujo, \emph{Rational curves of minimal degree and characterizations of
  projective spaces}, Math.\ Ann.\ \textbf{335} (2006), 937--951.

\bibitem[AW97]{AWaview}
M.~Andreatta and J.~A. Wi{\'s}niewski, \emph{A view on contractions of higher
  dimensional varieties}, Algebraic Geometry - Santa Cruz 1995,
  Proc.~Symp.~Pure Math., vol.~62, 1997, pp.~153--183.

\bibitem[BCHM10]{BCHM}
C.~Birkar, P.~Cascini, C.~D. Hacon, and J.~McKernan, \emph{Existence of minimal
  models for varieties of log general type}, J.\ Amer.\ Math.\ Soc.\
  \textbf{23} (2010), 405--468.

\bibitem[BCW02]{bonwisncamp}
L.~Bonavero, F.~Campana, and J.~A. Wi{\'s}niewski, \emph{Vari{\'e}t{\'e}s
  projectives complexes dont l'{\'e}clat{\'e}e en un point est de {F}ano},
  C.~R., Math., Acad.~Sci.~Paris \textbf{334} (2002), 463--468.

\bibitem[BS95]{beltrametti_sommese}
M.~C. Beltrametti and A.~J. Sommese, \emph{The adjunction theory of complex
  projective varieties}, de Gruyter Expositions in Mathematics, vol.~16, Walter
  de Gruyter \& Co., Berlin, 1995.

\bibitem[Cas09]{31}
C.~Casagrande, \emph{On {F}ano manifolds with a birational contraction sending
  a divisor to a curve}, Michigan Math.\ J.\ \textbf{58} (2009), 783--805.

\bibitem[CF11]{cilibertoflamini}
C.~Ciliberto and F.~Flamini, \emph{On the branch curve of a general projection
  of a surface to a plane}, Trans.\ Amer.\ Math.\ Soc.\ \textbf{363} (2011),
  3457--3471.

\bibitem[CFH14]{fuhwang}
Y.~Chen, B.~Fu, and J.-M. Hwang, \emph{Minimal rational curves on complete
  toric manifolds and applications}, Proc.\ Edinb.\ Math.\ Soc.\ \textbf{57}
  (2014), 111--123.

\bibitem[CMSB02]{cho}
K.~Cho, Y.~Miyaoka, and N.~Shepherd-Barron, \emph{Characterizations of
  projective space and applications to complex symplectic geometry}, Higher
  Dimensional Birational Geometry, Adv.\ Stud.\ Pure Math., vol.~35, Math.\
  Soc.\ Japan, 2002, pp.~1--89.

\bibitem[Cuk99]{cukierman}
F.~Cukierman, \emph{Monodromy of projections}, Mat.\ Contemp.\ \textbf{16}
  (1999), 9--30, 15th School of Algebra (Portuguese) (Canela, 1998).

\bibitem[Deb01]{debarreUT}
O.~Debarre, \emph{Higher-dimensional algebraic geometry}, Universitext,
  Springer-Verlag, 2001.

\bibitem[Dru04]{druel04}
S.~Druel, \emph{Caract\'erisation de l'espace projectif}, Manuscripta Math.\
  \textbf{115} (2004), 19--30.

\bibitem[FLS11]{FLS}
M.~Friedman, M.~Leyenson, and E.~Shustin, \emph{On ramified covers of the
  projective plane {I}: Interpreting {S}egre's theory (with an appendix by
  {E}ugenii {S}hustin)}, Internat.\ J.\ Math.\ \textbf{22} (2011), 619--653.

\bibitem[Fuj75]{fujita75}
T.~Fujita, \emph{On the structure of polarized varieties with {$\Delta
  $}-genera zero}, J. Fac. Sci. Univ. Tokyo Sect. IA Math. \textbf{22} (1975),
  103--115.

\bibitem[Fuj12]{kentofujita}
K.~Fujita, \emph{{F}ano manifolds having {$(n-1,0)$}-type extremal rays with
  large {P}icard number}, preprint arXiv:1212.4977, 2012.

\bibitem[GP13]{grusonpeskine}
L.~Gruson and C.~Peskine, \emph{On the smooth locus of aligned {H}ilbert
  schemes, the {$k$}-secant lemma and the general projection theorem}, Duke
  Math.\ J.\ \textbf{162} (2013), 553--578.

\bibitem[Har77]{hartshorne}
R.~Hartshorne, \emph{Algebraic geometry}, Graduate Texts in Mathematics,
  vol.~52, Springer-Verlag, 1977.

\bibitem[HM04]{hwang_mok_birationality}
J.-M. Hwang and N.~Mok, \emph{Birationality of the tangent map for minimal
  rational curves}, Asian J.\ Math. \textbf{8} (2004), 51--64.

\bibitem[Hwa01]{hwangICTP}
J.-M. Hwang, \emph{Geometry of minimal rational curves on {F}ano manifolds},
  School on {V}anishing {T}heorems and {E}ffective {R}esults in {A}lgebraic
  {G}eometry ({T}rieste, 2000), ICTP Lect. Notes, vol.~6, Abdus Salam Int.
  Cent. Theoret. Phys., Trieste, 2001, pp.~335--393.

\bibitem[Hwa10]{hwang_equivalence}
\bysame, \emph{Equivalence problem for minimal rational curves with isotrivial
  varieties of minimal rational tangents}, Ann.~Sci.~{\'E}c.~Norm.~Sup{\'e}r.\
  \textbf{43} (2010), 607--620.

\bibitem[Hwa13]{hwangcodone}
\bysame, \emph{Varieties of minimal rational tangents of codimension 1},
  Ann.~Sci.~{\'E}c.~Norm.~Sup{\'e}r.\ \textbf{46} (2013), 629--649.

\bibitem[IP99]{FanoEMS}
V.~A. Iskovskikh and Yu.~G. Prokhorov, \emph{Algebraic geometry {V} - {F}ano
  varieties}, Encyclopaedia Math.\ Sci.\, vol.~47, Springer-Verlag, 1999.

\bibitem[Ish91]{ishii}
S.~Ishii, \emph{Quasi-{G}orenstein {F}ano 3-folds with isolated non-rational
  loci}, Compos.\ Math.\ \textbf{77} (1991), 335--341.

\bibitem[KC06]{kebekussola}
S.~Kebekus and L.~Sol{\'a} Conde, \emph{Existence of rational curves on
  algebraic varieties, minimal rational tangents, and applications}, Global
  aspects of complex geometry, Springer, Berlin, 2006, pp.~359--416.

\bibitem[Keb02a]{kebekus}
S.~Kebekus, \emph{Characterizing the projective space after {C}ho, {M}iyaoka
  and {S}hepherd-{B}arron}, Complex Geometry (G{\"o}ttingen, 2000), Springer,
  2002, pp.~147--156.

\bibitem[Keb02b]{kebekusJAG}
\bysame, \emph{Families of singular rational curves}, J.\ Algebraic Geom.\
  \textbf{11} (2002), 245--256.

\bibitem[KM98]{kollarmori}
J.~Koll{\'a}r and S.~Mori, \emph{Birational geometry of algebraic varieties},
  Cambridge Tracts in Mathematics, vol. 134, Cambridge University Press, 1998.

\bibitem[Kol96]{kollar}
J.~Koll{\'a}r, \emph{Rational curves on algebraic varieties}, Ergebnisse der
  Mathematik und ihrer Grenzgebiete, vol.~32, Springer-Verlag, 1996.

\bibitem[Laz04]{lazI}
R.~Lazarsfeld, \emph{Positivity in algebraic geometry {I}}, Springer-Verlag,
  2004.

\bibitem[Miy04]{miyaokaQ}
Y.~Miyaoka, \emph{Numerical characterizations of hyperquadrics}, Complex
  Analysis in Several Variables - {M}emorial Conference of {K}iyoshi {O}ka's
  Centennial Birthday, Adv.\ Stud.\ Pure Math., vol.~42, Math.\ Soc.\ Japan,
  2004, pp.~209--235.

\bibitem[Tsu06]{toru}
T.~Tsukioka, \emph{Classification of {F}ano manifolds containing a negative
  divisor isomorphic to projective space}, Geom.~Dedicata \textbf{123} (2006),
  179--186.

\bibitem[Wi{\'s}91]{wisn}
J.~A. Wi{\'s}niewski, \emph{On contractions of extremal rays of {F}ano
  manifolds}, J.~Reine Angew.~Math.\ \textbf{417} (1991), 141--157.

\end{thebibliography}
\providecommand{\bysame}{\leavevmode\hbox to3em{\hrulefill}\thinspace}
\providecommand{\MR}{\relax\ifhmode\unskip\space\fi MR }
\providecommand{\MRhref}[2]{%
  \href{http://www.ams.org/mathscinet-getitem?mr=#1}{#2}
}
\providecommand{\href}[2]{#2}

\end{document}